\documentclass{article} 
\usepackage{iclr2024_conference,times}
\usepackage{soul}
\usepackage{float}
\usepackage{graphicx} 
\usepackage{subcaption}
\usepackage{booktabs}
\usepackage[linesnumbered, ruled]{algorithm2e}
\usepackage{appendix}

\usepackage{amsmath,amsfonts,bm}

\def\eqref#1{equation~\ref{#1}}

\def\1{\bm{1}}

\def\va{{\bm{a}}}
\def\vb{{\bm{b}}}

\def\vl{{\bm{l}}}

\def\vs{{\bm{s}}}
\def\vt{{\bm{t}}}

\def\vv{{\bm{v}}}

\def\vx{{\bm{x}}}

\DeclareMathAlphabet{\mathsfit}{\encodingdefault}{\sfdefault}{m}{sl}
\SetMathAlphabet{\mathsfit}{bold}{\encodingdefault}{\sfdefault}{bx}{n}

\DeclareMathOperator*{\argmax}{arg\,max}
\DeclareMathOperator*{\argmin}{arg\,min}

\usepackage{hyperref}
\usepackage{url}
\RequirePackage{amsthm,amsmath,amsfonts,amssymb}

\title{ProGO: Probabilistic Global Optimizer}

\author{Xinyu Zhang \& Sujit Ghosh \\ 
Department of Statistics\\
NC State University\\
Raleigh, NC 27695, USA \\
\texttt{\{xinyu$\_$zhang,sujit.ghosh\}@ncsu.edu} \\
}

\newcommand*\dd{\mathop{}\!\mathrm{d}}

\iclrfinalcopy 

\newtheorem{theorem}{Theorem}
\newtheorem{assumption}{Assumption}
\newtheorem{definition}{Definition}

\newtheorem{remark}{Remark}

\begin{document}

\maketitle

\begin{abstract}
In the field of global optimization, many existing algorithms face challenges posed by non-convex target functions and high computational complexity or unavailability of gradient information. These limitations, exacerbated by sensitivity to initial conditions, often lead to suboptimal solutions or failed convergence. This is true even for Metaheuristic algorithms designed to amalgamate different optimization techniques to improve their efficiency and robustness. To address these challenges, we develop a sequence of multidimensional integration-based methods that we show to converge to the global optima under some mild regularity conditions. Our probabilistic approach does not require the use of gradients and is underpinned by a mathematically rigorous convergence framework anchored in the nuanced properties of nascent optima distribution. In order to alleviate the problem of multidimensional integration, we develop a latent slice sampler that enjoys a geometric rate of convergence in generating samples from the nascent optima distribution, which is used to approximate the global optima. The proposed Probabilistic Global Optimizer (ProGO) provides a scalable unified framework to approximate the global optima of any continuous function defined on a domain of arbitrary dimension. Empirical illustrations of ProGO across a variety of popular non-convex test functions (having finite global optima) reveal that the proposed algorithm outperforms, by order of magnitude, many existing state-of-the-art methods, including gradient-based, zeroth-order gradient-free, and some Bayesian Optimization methods, in term regret value and speed of convergence. It is, however, to be noted that our approach may not be suitable for functions that are expensive to compute.

\end{abstract}

\section{Introcution} \label{sec:intro}
Global optimization constitutes a critical research area within applied mathematics and numerical analysis, aiming to locate the global optima of target functions over a specified domain. This field has substantial applications across various sectors in machine learning, such as hyperparameter tuning \citep{snoek2012practical}, signal processing \citep{liu2020primer}, and black-box adversarial attacks \citep{ru2019bayesopt}. A global optimization problem (minimization) with unique minima can be formulated as
\begin{equation} \label{eq:prob1}
   \vx^* = \argmin_{\vx \in \Omega} f(\vx), 
\end{equation}
where $f(\vx)$ is generally a continuous function defined over a domain $\Omega$ $(\subseteq\mathbb{R}^d)$, that is a subset of the $d$-dimensional Euclidean space $\mathbb{ R}^d$. Extensive progress has been made in optimizing globally convex functions over compact domains, where the global minima $\vx^*$ is guaranteed to be identified. However, less generalizable solutions exist for non-convex functions or on non-compact sets, even when some target function possesses smoothness or differentiability.

For semantic precision, we differentiate between ``optimum" / ``minimum", the optimal / lowest function value $f^*$, and ``optima" / ``minima", which corresponds to the $\vx^*$ at which $f^*$ is attained. In this paper, we assume the existence of a finite $f^*=\min_{\vx\in\Omega} f(\vx)$ and a non-empty set of minima $\Omega^*=\{\vx\in\Omega: f(\vx)=f^*\}$. When the minima $\vx^*$ in eq. (\ref{eq:prob1}) is unique, $\Omega^*$ will be a singleton set. Importantly, $\Omega^*$ can comprise either a finite or infinite number of elements. One of our main contributions lies in the identification of the limitations of gradient-based techniques, particularly for non-convex functions, and the introduction of a reliable, integration-based alternative that guarantees to locate the global optima without convex assumptions.  

Additionally, the efficacy of many global optimization algorithms is sensitive to initial points. Even metaheuristic algorithms, which amalgamate various optimization techniques for robustness, can yield suboptimal outcomes with poorly chosen initial conditions. Our method, under mild regularity conditions, is robust to initial conditions and yields accurate estimates of \(x^*\) within a decent computational timeframe, provided that function evaluations are not expensive. 


\paragraph{Gradient-based algorithms.} Gradient-based methods like stochastic gradient descent \citep{robbins1951stochastic}, Adam \citep{KingBa15}, AdaGrad \citep{duchi2011adaptive}, and RMSprop \citep{tieleman2012rmsprop} have found broad applications across disciplines. They have demonstrated their utility in various successful applications such as generative adversarial networks \citep{seward2018first} and reinforcement learning \citep{mnih2016asynchronous}. While these methods offer practical effective solutions, their theoretical convergence to global optima is often framed within specific contexts, particularly when the target function $f(\vx)$ is smooth and convex. Recent variants like AMSGrad \citep{reddi2019convergence} and parameter-selective approaches \citep{shi2020rmsprop} explore parameter effects on theoretical guarantee and practical efficacy.

\paragraph{Zeroth order (ZO) methods.} In cases where gradient information is unavailable, noisy, or computationally expensive to evaluate, such as signal processing and machine learning \citep{liu2020primer}, ZO methods have emerged driven by the need to solve these problems. These techniques, also known as ``black-box" or ``derivative-free" optimization, bypass the need for gradients and focus solely on function values at any given point \citep{larson2019derivative, rios2013derivative}. Several noteworthy methods have been proposed in this vein, including Gradientless Descent (GLD) \citep{golovin2019gradientless} which is numerically stable via a geometric approach, Random Gradient-Free (RGF) method via finite difference along a random direction by \citet{nesterov2017random}, and Prior-Guided Random Gradient-Free (PRGF) by \citet{cheng2021convergence}, typically operating under a framework that assumes convexity in the target functions. Additionally, \citet{shu2022zeroth} introduced the Zeroth-Order Optimization with the Trajectory-Informed Derivative Estimation (ZoRD) algorithm, further enriching the query-efficient ZO optimization methods landscape.

\paragraph{Global optimization challenges.} While gradient-based methods and ZO methods offer advantages, each comes with its set of limitations and is often contingent upon wise selections of initial parameters and starting points. Besides, Monte Carlo-based methods provide consistent global convergence \citep{harrison2010introduction} but can be computationally demanding in high-dimensional spaces. Bayesian optimization techniques such as the Gaussian Process-Upper Confidence Bound (GP-UCB) algorithm \citep{srinivas2009gaussian} and Trust Region Bayesian Optimization (TuRBO) \citep{eriksson2019scalable} assume $f(\vx)$ follows the Gaussian Process, whose performances hinge upon careful selection of acquisition and kernel functions. Analytical methods do contribute to domain-specific solutions but can entail intricate numerical challenges, affecting their widespread applicability \citep{corriou2021analytical}. Most literature in this field has been oriented towards establishing first-order optimality conditions, often under function convexity and differentiability assumptions. A notable work by \citet{luo2018minima} formalized a rigorous mathematical relation between an arbitrary continuous function $f$ defined over a compact set $\Omega \subseteq \mathbb{R}^d$ and its corresponding global minima $f^*$; however, this work only built a theoretical framework. This underscores the pressing need for an efficient and robust global optimization framework, especially in addressing non-convex and high-dimensional challenges.

\begin{figure}[t] 
    \centering
    \begin{minipage}[t]{.28\linewidth}
        \centering
        \subcaptionbox{Surface of Ackley ($d=2$) }{ 
        \includegraphics[width=\textwidth, trim={0 -170 0 0}, clip]{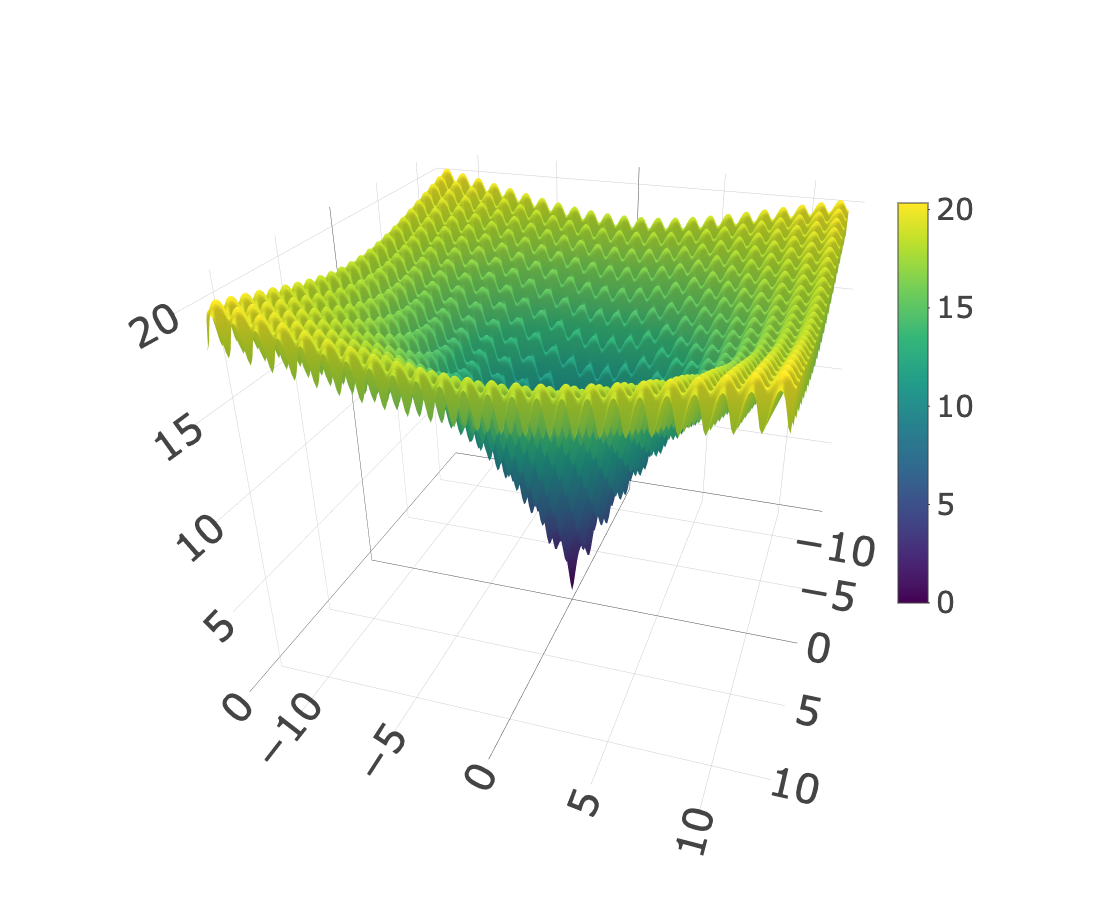}
        }
    \end{minipage}
    \begin{minipage}[t]{.7\linewidth}
        \centering
        \subcaptionbox{Results on Ackley ($d=1000$)}{ 
        \includegraphics[width=\textwidth]{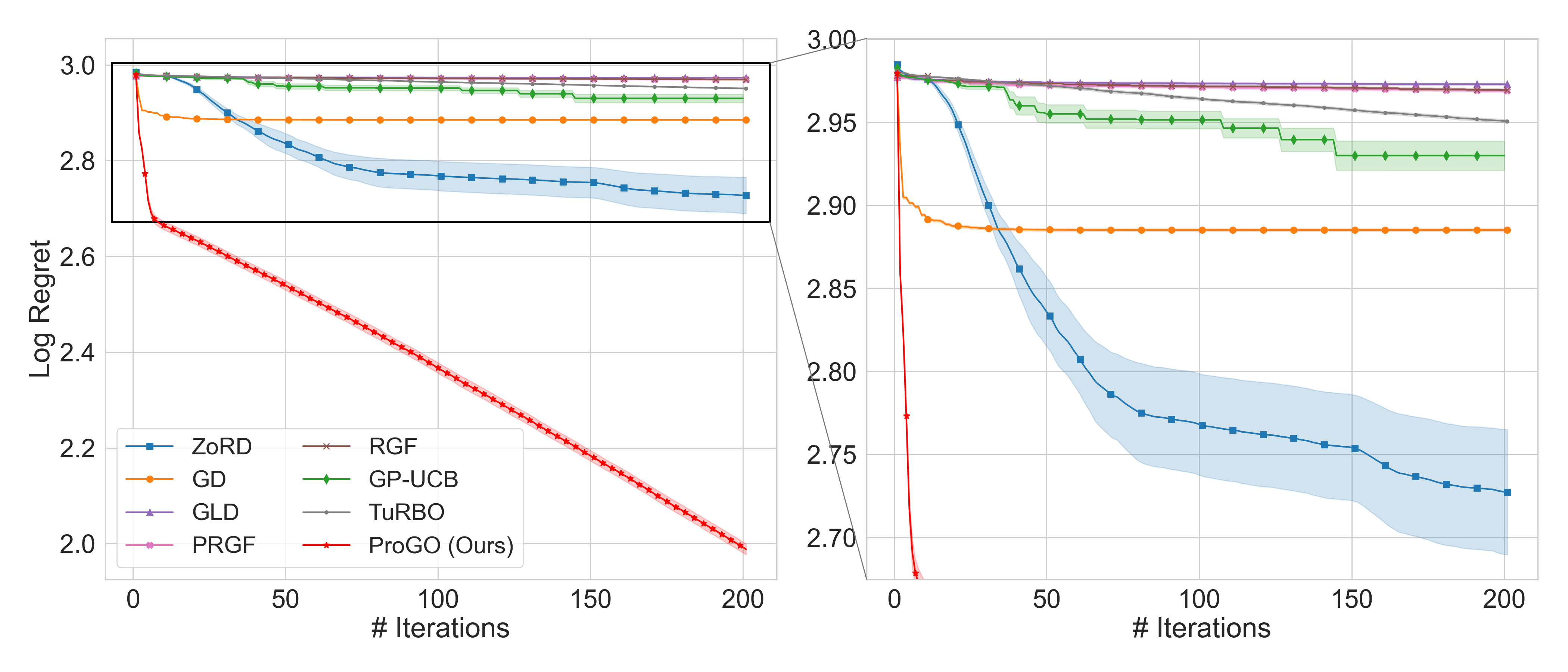}
        }
    \end{minipage}
    \caption{(a) Visualization of the Ackley function for $d=2$. (b) Evaluation of ProGO and competing methods (elaborated upon in Section \ref{sec:exp}) applied to the Ackley function in a high-dimensional setting with $d=1000$. The $x$-axis represents the iteration count, while the $y$-axis denotes the average log-scaled regret. Each curve shows the mean $\pm$ standard error across ten independent runs. ProGO outperforms other methods by a notably faster rate of convergence accompanied by smaller variability. \label{fig:ackley1000d}}
\end{figure}

\paragraph{Main contribution.} This paper introduces the Probabilistic Global Optimizer (ProGO), a novel non-gradient-based global optimization algorithm based on a sequence of sampling from a suitable probability distribution. Our work significantly extends the theoretical framework laid by \citet{luo2018minima}, notably in three key dimensions:

\begin{enumerate}
    \item \textbf{Generalization to Non-Compact Set:} \citet{luo2018minima}'s work is based on the assumption that $\Omega$ is a compact set. Such an assumption may limit the scope of its applicability to a class of popular functions when $\Omega=\mathbb{R}^d$. E.g., even when $d=1$, the elementary function $f(\vx)=\vx^2$ defined over $\mathbb{R}$ is unbounded. We generalize the domain to an arbitrary subset of $\mathbb{R}^d$ by incorporating a sequence of probability distribution on $\mathbb{R}^d$ that assigns less and less mass outside $\Omega^*$ and eventually converging with support $\Omega^*$.
    \item \textbf{A Practically Efficient Algorithm:} While a theoretical framework is established in \citet{luo2018minima} with several interesting results when the domain $\Omega$ is compact and the function $f$ is assumed continuous or smooth (with second order derivatives), to the best of our knowledge, no practical methods are yet available to evaluate the (potentially high-dimensional) integration required to estimate the minima. We fill this gap by developing a practically efficient algorithm that we call ProGO, which uses a latent slice sampler (explained later) to efficiently obtain samples from the probability distribution of the minima.
    \item \textbf{Extensive Experiments Validation:} We carried out a comprehensive series of experiments to evaluate ProGO's performance in comparison to various types of leading global optimization techniques, including Gradient Descent (GD), Zeroth-Order Optimization (ZO), and Bayesian Optimization (BO). Our empirical evidence demonstrates that ProGO consistently surpasses all the algorithms we compared against across various metrics -- most notably, geometric rate of convergence to global minima and computational efficiency as indicated by function evaluations or CPU time. Specifically, we illustrate the superior numerical performance of our proposed ProGO for the popular Ackley function (known to have several local optima with a global minimum at the origin) for dimensions ranging from $d=20$ to $d=1000$ (refer to Fig. \ref{fig:ackley1000d}), as well as the Levy function (refer to Fig. \ref{fig:levy1000d}). As depicted in the figures, the logarithmic regret exhibits a linear rate of convergence, which in the original scale translates to geometric convergence, outpacing the majority of extant global optimization algorithms.
\end{enumerate}

The rest of this paper unfolds as follows: Section \ref{sec:prelim} lays the theoretical groundwork on the probabilistic minima distribution for our approach; Section \ref{sec:progo} details the latent slice sampler and our ProGO algorithm; Section \ref{sec:exp} presents the empirical validations, and Section \ref{sec:conc} concludes the paper.

\section{Preliminaries} \label{sec:prelim}

\subsection{Nascent Minima Probability Distribution}

Our proposed algorithm, ProGO, is based on generalizing the sequence of nascent minima (probability) distributions defined in \citet{luo2018minima} using an arbitrary (prior) probability measure with full support on the Euclidean space $\mathbb{R}^d$. This distribution possesses advantageous properties that will be elaborated upon in subsequent discussions. In particular, we will show how to efficiently generate samples from the sequence of such distributions and subsequently use the empirical (posterior) mean and other summaries to estimate the minimum value $f^*$ and the minima $\vx^*$.

\begin{assumption} \label{def:1}
The following conditions are assumed throughout the paper:
\begin{enumerate}
    \item[(i)] Assume that the function $f:\Omega\subseteq\mathbb{R}^d \rightarrow \mathbb{R}$ is a continuous function with a finite global minimum value $f^*$; i.e., $f(\vx)\geqslant f^*$ for all $\vx\in\Omega$.
    \item[(ii)] The set of global minima $\Omega^* =\{\vx\in \Omega: f(\vx)=f^*\}$ is non-empty.
    \item[(iii)] There is a probability measure with density $\pi(\vx)$ that has full suppprt on $\mathbb{R}^d$. In other words, $\pi(\vx)>0$ for any $\vx\in\mathbb{R}^d$ and $\int_{\mathbb{R}^d} \pi(\vx) d\vx = 1$. Here, the integration is with respect to the Lebesgue measure on $\mathbb{R}^d$.
\end{enumerate}
\end{assumption}

In the above, the probability density $\pi(\vx)$ can be chosen arbitrarily, but in practice, we can use a uniform distribution when $\Omega$ is compact, or a very flat (nearly uniform) distribution when $\Omega$ is unbounded. Regardless, we next define a nascent minima distribution when it depends on the choice of the density $\pi(\cdot)$.

\begin{definition} \label{def:minima}
Nascent Minima (probability) distribution:

For any $k\geqslant 0$, a nascent minima distribution density is defined as:
\begin{align} \label{def:nmd}
    m_k(\vx) = \frac{e^{-kf(\vx)} \cdot \pi(\vx) }{\int_{\Omega} e^{-kf(\vt) } \cdot \pi(\vt) \dd \vt}.
\end{align}
\end{definition}

\begin{remark} \label{rmk:1} Note that the denominator in eq. (\ref{def:nmd}) is a finite positive quantity for any arbitrary $k\geqslant 0$, because $0<\int_{\Omega} e^{-kf(\vt) } \cdot \pi(t) \dd t\leqslant e^{-k f^*}$. The assumption that $\pi(\cdot)$ is a probability density can be relaxed as long as $\int_{\Omega} e^{-kf(\vt) } \cdot \pi(\vt) \dd \vt<\infty$ for any $k>0$, even when $\int_{\Omega} \pi(t) \dd t =\infty$.
\end{remark}

Notice that $m_k(\vx)$ can also written as
\begin{align}  
    m_k(\vx) & 
     = \frac{e^{-k (f(\vx)-f^*) } \cdot \pi(\vx) }{\int_{\Omega} e^{-k (f(\vt)-f^*) } \cdot \pi(\vt) \dd \vt }.
\end{align}
By replacing the original $f(\vx)$ by $f(\vx) - f^*$, we may assume without loss of generality that $f$ is a non-negative valued function and has a global minimum $f^{*}= 0$. Consequently, we focus on finding a solution in the set of global minima $\Omega^* = \{\vx \in \Omega: f(\vx) = 0\} \neq \emptyset$ for the rest of this paper for all subsequent theoretical analyses. 

However, it should be noted that for practical applications, we need to work with the original function $f$ as we will not know that global minimum $f^*$ and set of minima $\Omega^*$, and our goal would be to approximate $f^*$ and $\vx^*\in\Omega^*$ by letting $k\rightarrow\infty$.

\begin{remark} \label{rmk:2}
If the original $f$ is a positive valued function with $f^*>0$, we can replace it with $\log f$ when defining the nascent minima density in eq. (\ref{def:nmd}). Also, if a global maximum is desired, we can replace the original $f$ by $-f$ in defining the nascent minima density in eq. (\ref{def:nmd}).
\end{remark}

Next, we provide two results under very minimal conditions, which establish the convergence of the (generalized) moments of the nascent minima distribution to the minimum value $f^*$. With additional conditions, we also establish the convergence of the minima.


\subsection{Convergence Properties}
In this subsection, we establish the theoretical underpinnings regarding the convergence properties of the nascent minima distribution.

\begin{theorem}[] \label{thm:1}
Consider a function $f$ and a probability density $\pi$ satisfying the assumptions (i)-(iii) given in Assumption \ref{def:1}. Then, the nascent minima distribution has the following properties:
\begin{equation}
   \lim _{k \rightarrow \infty} \int_{\Omega} f(\vx) m_k(\vx) \mathrm{d} \vx=f^*  
   \label{eq:conv}
\end{equation}
\end{theorem}

The proof is detailed in Section \ref{sec:proof1}. The above result implies that the expected value $\mathbb{E}_{m_k}[f(X)]$ converges to the minimum value $f^*$ and hence, if we are able to generate samples from the nascent minima distribution $m_k(\cdot)$ for any $k>0$, then we can approximate $f^*$ arbitrarily close by choosing a large $k>0$.  Next, we show that the above convergence is monotonic, which in turn implies that by increasing $k$ sequentially, we will get closer and closer to the minimum value $f^*$.

\begin{theorem}[monotonicity] \label{thm:2}
Consider a non-constant function $f$ and a probability density $\pi$ satisfying the assumptions (i)-(iii) given in Assumption \ref{def:1}. For each $k>0$, let $\mu_k=\mathbb{E}_{m_k}[f(X)]$ denotes the expectation of $f(X)$ when $X\sim m_k(\cdot)$. Then the sequence $\{\mu_k\}$ is monotonically decreasing and satisfies
\begin{align}
\frac{\mathrm{d}\mu_k}{\mathrm{d} k} = \frac{\mathrm{d}}{\mathrm{d} k}\int_{\Omega} f(\vx) m_k(\vx) \mathrm{d} \vx=-\mathbb{V}{ar}_k(f) < 0, 
\end{align}
where $\mathbb{V}{ar}_k(f)=\int_{\Omega}\left(f(\vx)-\mu_k\right)^2 m_k(\vx) \mathrm{d} \vx$ denotes the variance of $f(X)$ when $X\sim m_k(\cdot)$.
\end{theorem}

The detailed proof is provided in Section \ref{sec:proof2}. Notice that the monotonic convergence of $\mu_k$ as $k$ increases is established for any continuous function $f$ and probability density $\pi$ satisfying (i)-(iii) of Assumption \ref{def:1}.  This allows us for a very general use of ProGO with minimal assumptions for any dimension $d\geqslant 1$. Next, we explore the convergence of the minimum values with additional assumptions.

Here, we introduce the strong separability condition to define the scenario in which this ProGo method is most suitable. 

\begin{assumption} [strong separability condition] \label{assumption2} Consider the set of minima $\Omega^*=\{\vx\in\Omega: f(\vx)=f^*\}$ which is assumed to be non empty. Then $f$ is said to satisfy a strong separability if, for any given $\delta>0$, we have $\inf_{\vx\notin\Omega^*: ||\vx-\tilde{\vx}||>\delta} f(\vx) > f^*$ for any $\tilde{\vx}\in\Omega^*$.
\end{assumption}

The above condition implies that if $\vx\notin\Omega^*$ and $||\vx-\tilde{\vx}||>\delta$ for some $\delta>0$, then $f(\vx)>f^*$. In other words, the $f$ values for $\vx$ not in $\Omega^*$ are well separated from those that are in $\Omega^*$ and hence $\epsilon_0=\inf\{f(\vx)-f^*: \vx\notin\Omega^*\}>0$.

\begin{theorem} \label{thm:3}
Consider a bounded probability density $\pi$ and a target function satisfying the assumptions (i)-(iii) in Assumption 1. For each $k>0$, let $\Omega_k = \{\vx \in \Omega: m_k(\vx) \geqslant m_k(\tilde{\vx}), \forall \tilde{\vx} \in \Omega\}$ denotes the set of maximizes for $m_k(\vx)$. Then for any sequence $\{\vx_k^*\}$, $\vx_k^* \in \Omega_k$, it satisfy

\begin{equation} \label{eq:thm31}
       \lim_{k \rightarrow \infty} f(\vx_k^*) =f^*,
\end{equation}


In addition, if the target function satisfies the strong separability condition in Assumption 2, we have 
\begin{equation} \label{eq:thm32}
       \lim_{k \rightarrow \infty} \inf_{ \tilde{\vx} \in \Omega^*} ||\vx_k^* - \tilde{\vx}|| = 0.
\end{equation}

\end{theorem}

The proof is available in Section \ref{sec:proof3}. This theorem establishes the convergence of the sequence ${\vx_k^*}$ toward the global minima set $\Omega^*$ by quantifying the metric $\min_{ \tilde{\vx} \in \Omega^*} ||\vx_k^* - \tilde{\vx}||$. This metric serves as a measure of divergence between the iteratively obtained minima $\vx_k^*$ for each $k$ and elements from the true global minima set $\Omega^*$. The empirical evaluations in Section \ref{sec:exp} corroborate the algorithm's efficacy in identifying discrepancies in both the optimal value and corresponding minima.

\section{ProGO: A New Probabilistic Global Optimization Method} \label{sec:progo}

This section describes Latent Slice Sampler (LSS) in Section \ref{sec:lss}, followed by the ProGO algorithm in Section \ref{sec:progoalgo}.

\subsection{Latent Slice Sampler} \label{sec:lss} Compared to traditional Markov Chain Monte Carlo techniques like the Metropolis–Hastings algorithm, slice sampling (SS) provides benefits such as reduced asymptotic variance and accelerated convergence \citep{mira2002efficiency, neal2003slice, roberts1999convergence}. From several SS variants such as Elliptical SS by \citet{murray2010elliptical} and Polar SS by \citet{rudolf2023dimension}, we adopt the LSS introduced by \citet{li2023latent} because it eliminates the requirement for proposal distribution and improves efficiency in high-dimensional sampling. 

For a detailed understanding of LSS, consider the target distribution as a minima distribution $m(\vx)$ for a $d$-dimensional variable $\vx$. By incorporating slice variables $w$, $\vs=(s_1, \cdots, s_d)^\top$, and $\vl=(l_1, \cdots, l_d)^\top$, the joint density can be formulated as follow:

$$
p(\vx, w, \vs, \vl) = \mathbb{I}\left(m(\vx)>w\right) p(\vs) \prod_{j=1}^d \frac{\mathbb{I}\left(\vx_j-\vs_j/2<\vl_j<\vx_j+\vs_j/2\right)}{\vs_j}.
$$
Each $\vs_j$ for $j=1, \cdots, d$ is assumed to follow an independent gamma distribution with a shape parameter of 2 and scale parameter $\beta$ of 20, following \citet{luo2018minima}. Let $\vx^{(t)}$ represent the sample obtained after $t^{th}$ iteration, the full LSS algorithm implemented via Gibbs sampling is presented in Algorithm \ref{alg:lss}.
 
\begin{algorithm}[t]
\caption{Latent Slice Sampler for ProGO (LSS-ProGO)\label{alg:lss}}
\KwIn{Target probability distribution $\pi: \Omega \rightarrow [0,1]$, sample size $N$, initialization $\vx^{(0)}$, burn-in period $n_b$}
Initialize $\vx=\vx^{(0)}$,  $w^{(0)} \sim U(0, \pi(\vx^{(0)}))$, $\vs^{(0)}_j \sim \mathrm{Gamma}(2, \beta)$, $\mathrm{for}$ $j=1, \cdots, d,$ $\vl^{(0)} \sim U(\vx^{(0)} - \vs^{(0)} / 2, \vx^{(0)} + \vs^{(0)} / 2)$ \\
 \For{$\mathrm{iteration}$ $t\in \{1, 2, \cdots, N + n_b\}$} {
      $\va \leftarrow \vl^{(t-1)}-\vs^{(t-1)}/2$ \\
      $\vb \leftarrow \vl^{(t-1)}+\vs^{(t-1)}/2$ \\
      \While{$\vx \notin \{\vx: \pi(\vx)>w^{(t-1)}\}$}{
          \For{ $\mathrm{dimension}$ $j \in \{1, \cdots, d\}$}{
            \eIf{$\vx < \vx_j$}{
              $\va_j \leftarrow \max(\va_j, \vx_j)$
              }{
              $\vb_j \leftarrow \min(\vb_j, \vx_j)$
              }
          }
        $\vx \sim U(\va, \vb)$ \\
      }
  $\vx^{(t)} \leftarrow \vx$ \\
  $w^{(t)} \sim U(0, \pi(\vx^{(t)}))$ \\
  $\vs^{(t)} \sim e^{\beta} + 2|\vl^{(t)} - \vx^{(t)}| $ \\
  $\vl^{(t)} \sim U(\vx^{(t)} - \vs^{(t)} / 2, \vx^{(t)} + \vs^{(t)} / 2)$ \\
 }

\KwOut{Samples $\{\vx^{(n_b+1)}, \cdots, \vx^{(n_b+N)}\}$ from target probability distribution $\pi$} 
\end{algorithm}

\paragraph{One Dimensional Illustration.}
To demonstrate the rationale behind using the nascent minima distribution for global optimization and to evaluate the efficacy of LSS, we consider a one-dimensional example adapted from \citet{luo2018minima}, given by $f(x)=\cos(x^2)+ x/5 +1,  x\in[0,5].$
This function possesses three local minima and a singular global minima at $f(x)=0.353$ when $x=1.756$.

\begin{figure}[h]
    \centering
    \begin{minipage}[t]{\linewidth}
		\centering
          \subcaptionbox{$k=0$}{ \label{fig:a}
          \includegraphics[width=.23\textwidth]{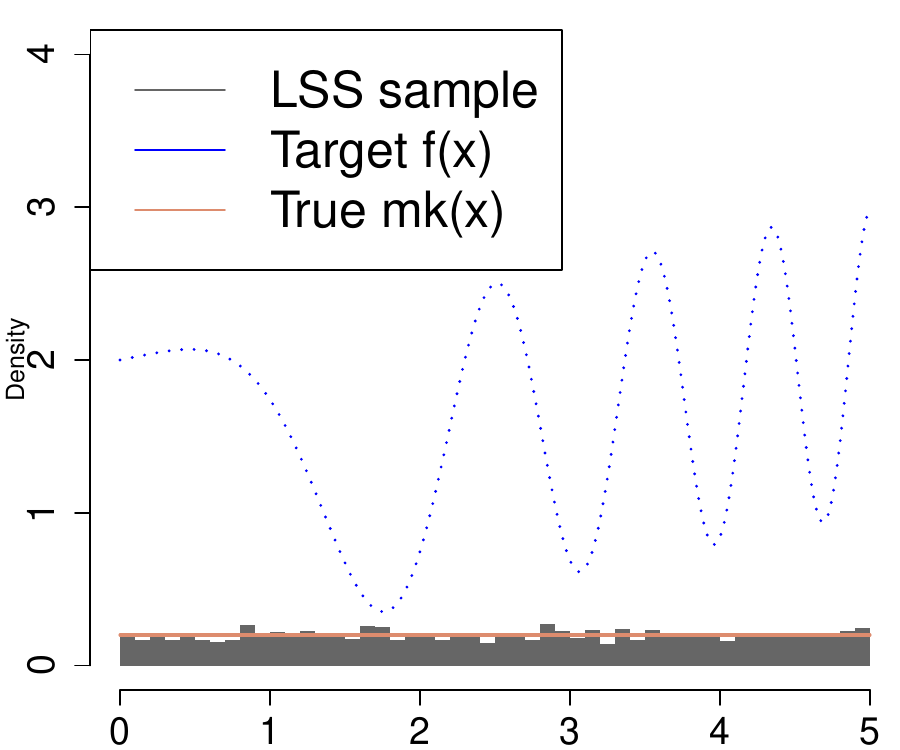}
          }
          \subcaptionbox{$k=1$}{
          \includegraphics[width=.23\textwidth]{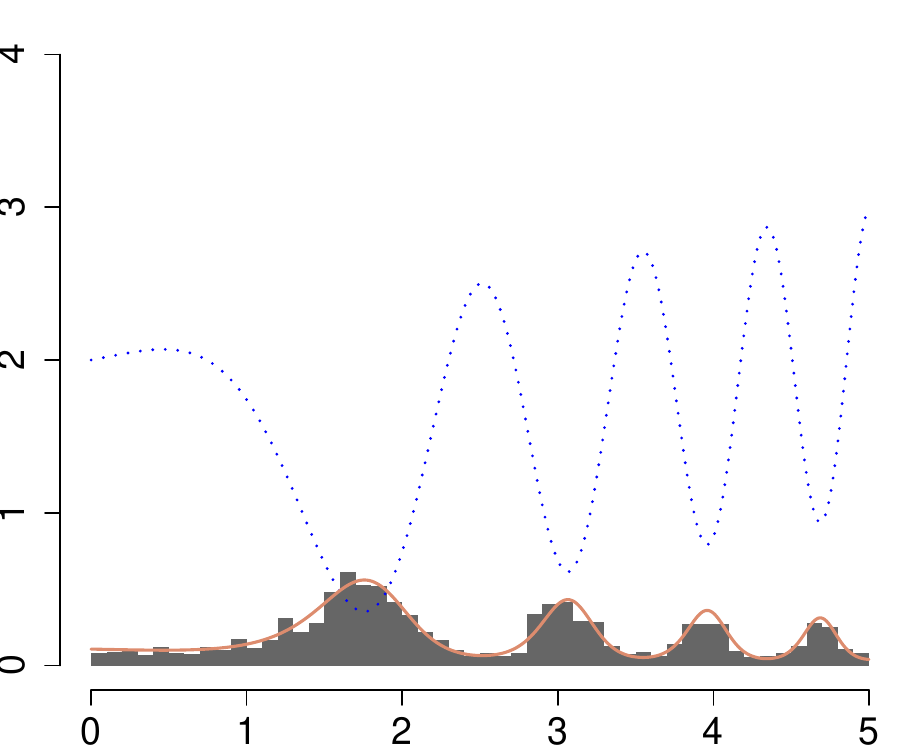}
          }
          \subcaptionbox{$k=3$}{
          \includegraphics[width=.23\textwidth]{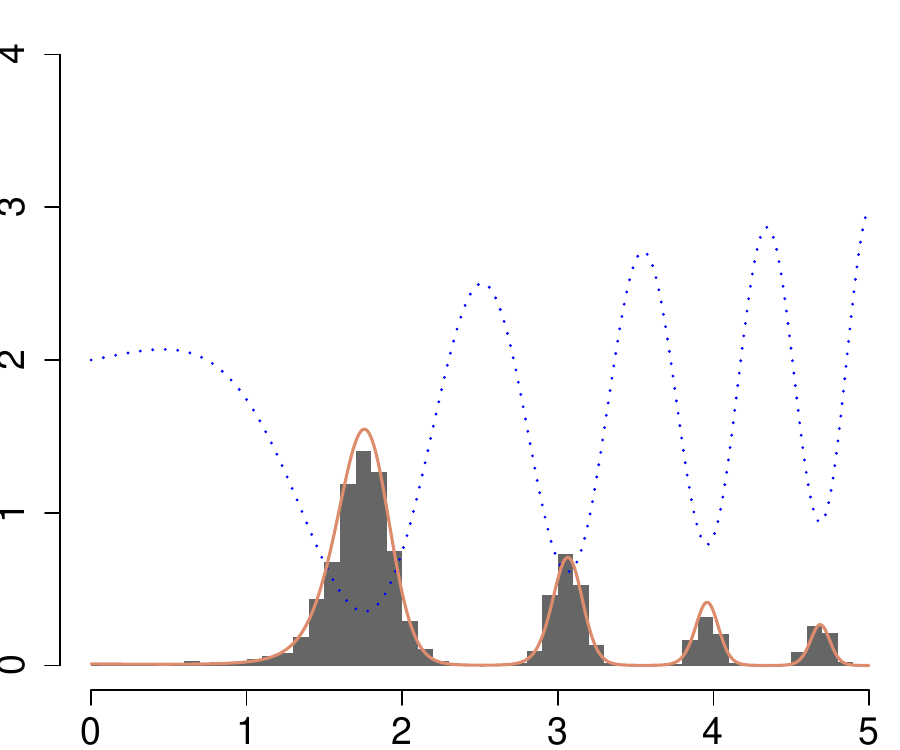}
          }
          \subcaptionbox{$k=9$}{
          \includegraphics[width=.23\textwidth]{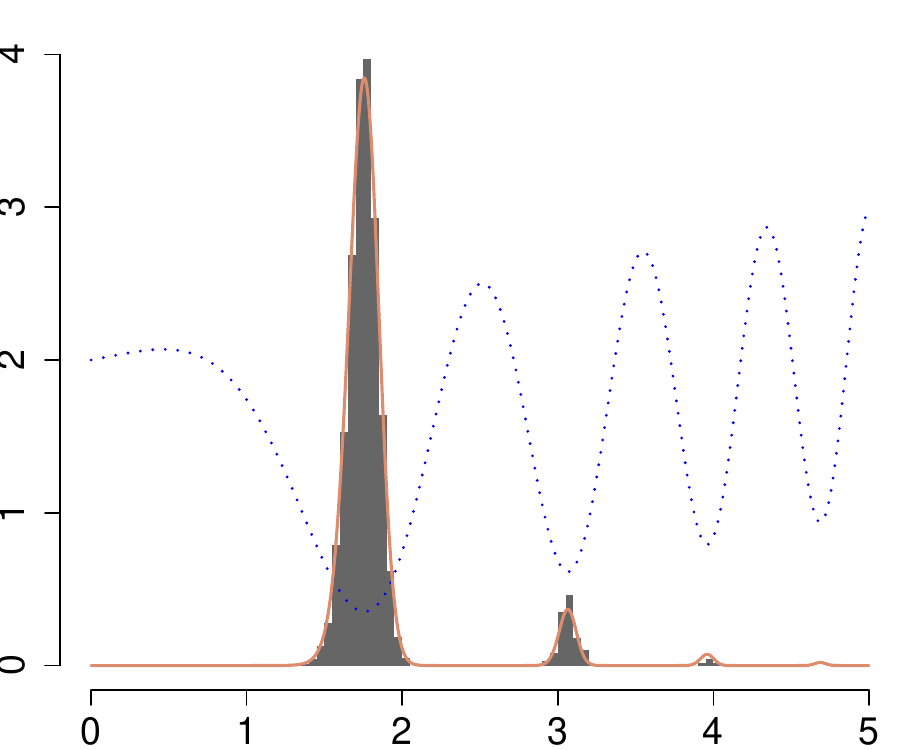}
          }
    \end{minipage}
    \caption{One dimensional illustration of nascent minima distribution $m_k(\vx)$ and its latent slice samples. In this illustration, $\pi(\vx)$ is assumed to satisfy Assumption \ref{def:1} and uniformly distributed within [0,5]. The target function $f(x)=\cos(x^2)+x/5+1$ is depicted in the blue curve. The true $m_k(\vx)$ with $k=0, 1, 3, 9$ are shown as a red curve. The black histograms represent the latent slice samples of $m_k(\vx)$. \label{fig:lsspbmd} }

\end{figure}

As depicted in Figure \ref{fig:lsspbmd}, the density distributions generated through LSS closely align with the actual minima distribution across various $k$. Notably, when $k=0$, $m_0(\vx) = \pi(\vx)$ corresponds to a uniform distribution, as shown in Figure \ref{fig:lsspbmd}(a). A higher value of $k$ leads to a minima distribution where the probability is increasingly focused on the global optima.

\subsection{ProGO} \label{sec:progoalgo}
We outline the algorithm of ProGO in Algorithm \ref{alg:progo}. The step size is incremented by $\triangle k$ for each iteration, chosen as $\triangle k=(e-1)k$. Such choice is motivated by the inverse shrinkage rate of $m_k(\vx)$ with respect to $k$, as demonstrated in \citet{luo2018minima}. The starting value of $k$ establishes the subsequent sequence $\{k, ke, ke^2, \cdots, ke^T\}$ across $T$ iterations, leading to a converging sequence of $\{f(\vx_k^*), f(\vx_{ke}^*), f(\vx_{ke^2}^*), \cdots, f(\vx_{ke^T}^*)\}$, where $x_k^* \in \{\vx \in \Omega: m_k(\vx) \geqslant m_k(\tilde{\vx}), \forall \tilde{\vx} \in \Omega\}$. Based on our preliminary results, the initial value of $k$ is set to 5.

\begin{algorithm}[t]
\caption{ProGO Algorithm}\label{alg:progo}
\KwIn{Target function $f: \Omega \rightarrow \mathbb{R}$, probability distribution function $\pi(\cdot)$, max iteration number $T=200$, sample size $N$, starting point $\vx^{(0)}$, burn-in period $n_b$}
Initialize $k=5$\\
 \For{$\mathrm{iteration}$ $t \in \{1, 2, \cdots, T\}$} {
 $\vx^{(1)}, \cdots, \vx^{(N)} \sim \text{LSS-ProGO}(m_k( \cdot~ ;\pi, f), N, \vx^{(0)}, n_b)$ \\
 $\tilde{\vx}^{(t)} \leftarrow \argmax_{i \in \{1, 2, \cdots, N\}} m_k(\vx^{(i)};\pi, f))$ \\
 $k \leftarrow k + (e-1) k$
 }
\KwOut{$\argmin_{t \in \{1, 2, \cdots, T\}} f(\tilde{\vx}^{(t)})$} 
\end{algorithm}

It is noteworthy that the output of ProGO delivers more than just an optimum value; it also provides the sample sets from the minima distribution. This provides valuable information about the distributional properties of local minima, as exemplified in Figure \ref{fig:lsspbmd}.

\section{Experiments} \label{sec:exp}
Test functions, evaluation metrics, and benchmarking methods play crucial roles in validating the performance of an optimization algorithm. While a wide array of test functions exists in the literature \citep{jamil2013literature}, the Ackley function—formulated initially by \citet{ackley2012connectionist}—and the Levy function remain prevalent choices, as corroborated by recent studies like \citet{shu2022zeroth}. 

As for evaluation criteria, we use the following metrics of function log regret and minima log regret to capture discrepancies in both $f(\vx)$ and $\vx$.
\begin{definition} \label{def:met}
Given $\tilde{\vx}$ as an estimated optima in a $d$-dimension space, the \textbf{function log regret} is defined as:
\begin{equation}
    r_f = \log \left(f(\tilde{\vx}) - f^* \right),
\end{equation}
quantifying the deviation between the estimated and true global optimum. The \textbf{minima log regret} is formulated as:
\begin{equation}
    r_m = \log   \frac{ || \tilde{\vx} - \vx^* || }{ \sqrt{d} } ,
\end{equation}
which quantifies the discrepancy between the estimated and true global minima. 
\end{definition}

The experimental design and the selection of competing algorithms are mostly aligned with the framework presented in ZoRD \citep{shu2022zeroth} to ensure a consistent and fair evaluation (details in Section \ref{sec:supexp}). A computational budget capped at 200 iterations is allotted for each of the ten independent runs conducted in every experiment setting. The evaluated methods include:
\textbf{1) ZoRD}: Zeroth-order trajectory-informed derivative estimation \citep{shu2022zeroth}. \textbf{2) GD}: Gradient-Descent, directly using first-order information. \textbf{3) GLD}: Gradientless Descent \citep{golovin2019gradientless}.  \textbf{4) PRGF}: Prior-guided random gradient-free algorithm \citep{cheng2021convergence}. \textbf{5) RGF}: Random gradient-free method via finite difference \citep{nesterov2017random}. \textbf{6) GP-UCB}: Gaussian Process-Upper Confidence Bound \citep{srinivas2009gaussian}. \textbf{7) TuRBO}: Trust Region Bayesian Optimization \citep{eriksson2019scalable}. \textbf{8) ProGO (our approach)}: Probabilistic Global Optimization.


\subsection{Ackley}
The Ackley function serves as a prominent benchmark for evaluating optimization algorithms. It is characterized as a continuous, differentiable, multimodal, and non-convex function, thereby posing significant optimization challenges. Mathematically, the $d$-dimensional Ackley function is given by:
\begin{equation} \label{eq:ackley}
    f({\vx})=-a \exp \left(-b \sqrt{\frac{1}{d} \sum_{i=1}^d \vx_i^2}\right)-\exp \left(\frac{1}{d} \sum_{i=1}^d \cos \left(c \vx_i\right)\right)+a+\exp (1),
\end{equation}
where recommended parameter values are $a=20, b=0.2$ and $c=2 \pi$ \citep{adorio2005mvf}. As shown in Figure \ref{fig:ackley1000d} (a), the Ackley function features numerous local minima and a central global minima at $f^* = 0$ and $\vx^* = (0, \cdots, 0)^\top$, presenting multiple local minima traps for optimization algorithms.
\begin{figure}[h]
    \centering
    \begin{minipage}[t]{\linewidth}
        \centering
        \subcaptionbox{Ackley (d=20)}
        {\includegraphics[width=\textwidth]{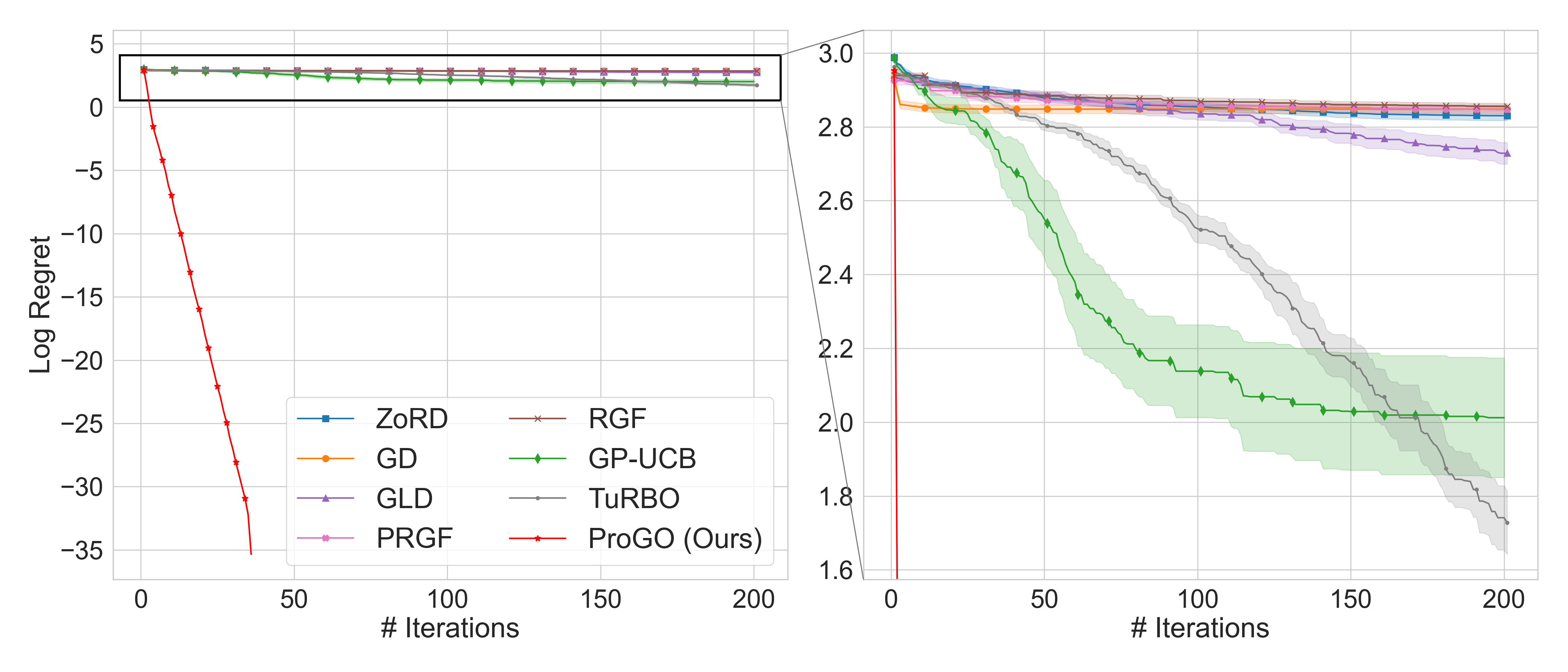}}
    \end{minipage}
    \begin{minipage}[t]{\linewidth}
        \centering
        \subcaptionbox{Ackley (d=40)}
        {\includegraphics[width=\textwidth]{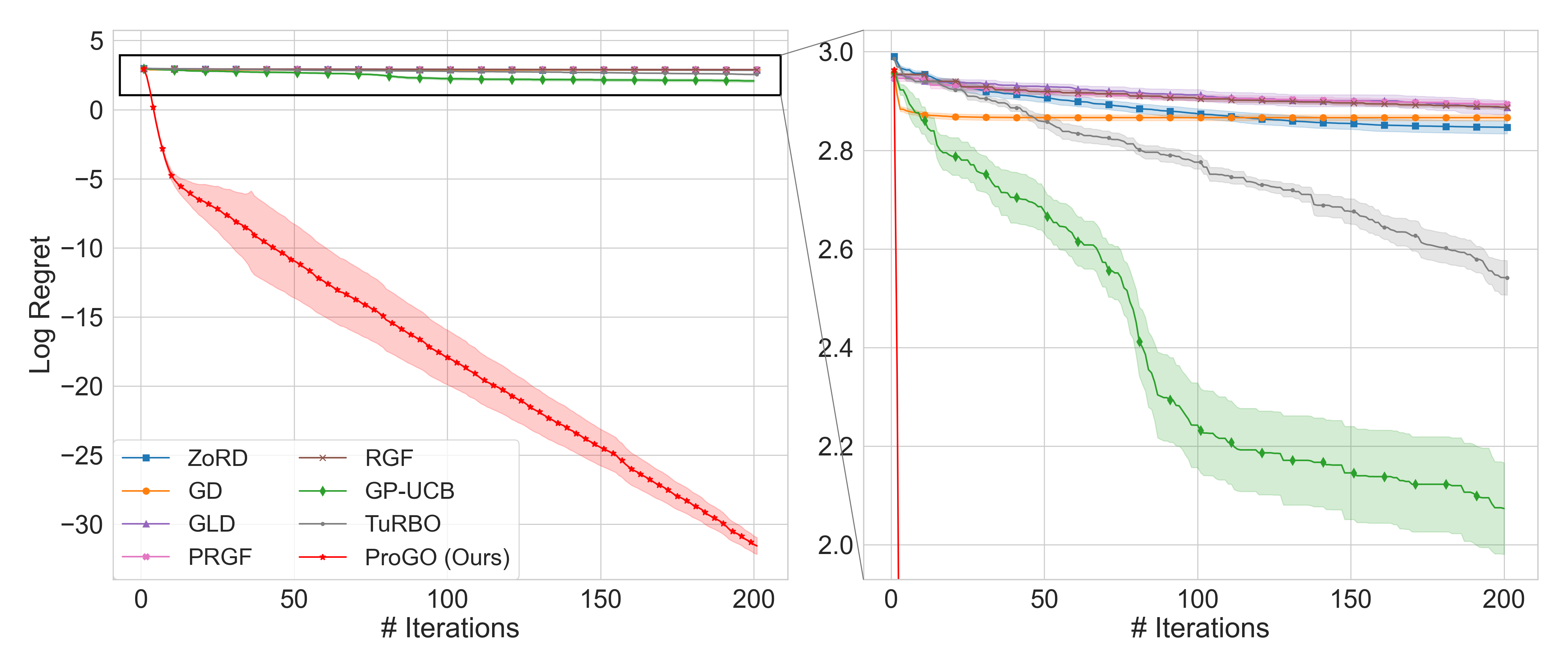}}
    \end{minipage}
\caption{Evaluation of ProGO and competing methods applied to the Ackley function with $d=20, 40$. The $x$-axis represents the iteration count, while the $y$-axis denotes the average log-scaled regret. Each curve shows the mean $\pm$ standard error across ten independent runs. \label{fig:ackley40d}}
\end{figure}

\begin{table}[H]
	\begin{center}
		\caption{Comparison of performance on the Ackley function for dimensions $d=20, 40, 1000$: Each optimization method has ten independent runs. Accuracy is quantified using the average function log regret ($r_f$) and the average minima log regret ($r_m$) across these ten runs. Computational efficiency is represented by the average runtime ($t$) across all ten runs, measured in seconds.\label{tab:ackleyexp}}
  \resizebox{\linewidth}{!}{  
    \begin{tabular}{@{}c rrr c rrr c rrr@{}} \toprule
        & \multicolumn{3}{c}{$d=20$} & & \multicolumn{3}{c}{$d=40$} &  & \multicolumn{3}{c}{$d=1000$} \\ \cmidrule{2-4} \cmidrule{6-8} \cmidrule {10-12} 
        \multicolumn{1}{c}{ Method}  &\multicolumn{1}{c}{\bf ${r_f}$} &\multicolumn{1}{c}{\bf ${r_m}$} 
        &\multicolumn{1}{c}{\bf ${t~ (\mathrm{s})}$}  & &\multicolumn{1}{c}{\bf ${r_f}$} &\multicolumn{1}{c}{\bf ${r_m}$} 
        &\multicolumn{1}{c}{\bf ${t~ (\mathrm{s})}$}  & &\multicolumn{1}{c}{\bf ${r_f}$} &\multicolumn{1}{c}{${r_m}$} 
        &\multicolumn{1}{c}{\bf ${t~ (\mathrm{s})}$}  \\ \midrule
  ZoRD   &            2.83 &          1.68 & 1265.43 & &             2.85 &           1.72 & 1092.04 & &             2.73 &           1.58 &  406.16\\
    GD   &           2.85 &          1.72 &     0.12 & &             2.87 &           1.75 &     0.13 & &             2.89 &           1.75 &     0.14\\
   GLD  &            2.73 &          1.65 &     0.38 & &             2.89 &           1.70 &     0.39 & &             2.97 &           1.75 &     0.40 \\
  PRGF   &            2.85 &          1.70 &     0.06 & &             2.89 &           1.74 &     0.07 & &             2.97 &           1.75 &     0.08\\
   RGF   &            2.86 &          1.66 &     0.06 & &             2.89 &           1.72 &     0.07 & &             2.97 &           1.74 &     0.08 \\
GP-UCB   &          2.01 &          1.62 &  132.13 & &             2.07 &           1.63 &  323.12 & &             2.93 &           1.74 & 1132.61 \\
 TuRBO   &            1.73 &          1.60 &   30.83 & &             2.54 &           1.62 &   83.84  & &             2.95 &           1.73 &  272.56\\
 \textbf{ProGO}   & \textbf{-35.35} &        \textbf{-35.86} &    6.77 & &           \textbf{-31.56} &          \textbf{-9.50} &   16.25 & &             \textbf{1.99} &           \textbf{0.52} &  280.05\\\bottomrule
		\end{tabular}
}
	\end{center}
    \end{table}

Our empirical evaluations (see Table \ref{tab:ackleyexp}) span dimensions $d=20, 40,$ and $1000$. Across all dimensions, ProGO consistently outperforms other methods, achieving significantly lower function log regret and lower minima log regret. Moreover, Figure \ref{fig:ackley1000d}(b) corroborates ProGO's geometric rate of convergence even on high dimensions (see Figure \ref{fig:ackley40d} for results on $d=20$ and $d=40$).

\begin{figure}[h]
	\centering
	\begin{minipage}[t]{\linewidth}
		\centering
  \subcaptionbox{Levy (d=40)}
      {\includegraphics[width=\textwidth]{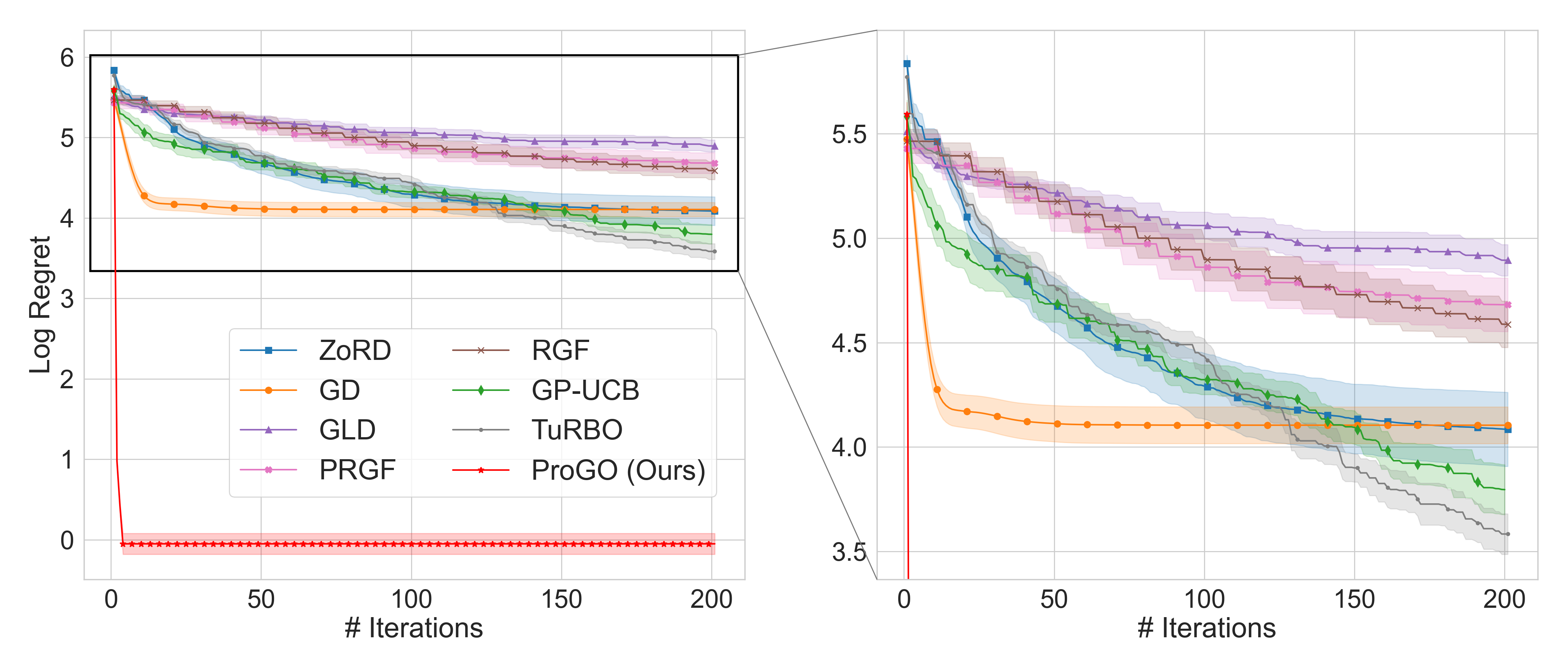}}
      
	\end{minipage}
	\centering
        \begin{minipage}[t]{\linewidth}
		\centering
		\subcaptionbox{Levy (d=100)}
      {\includegraphics[width=\textwidth]{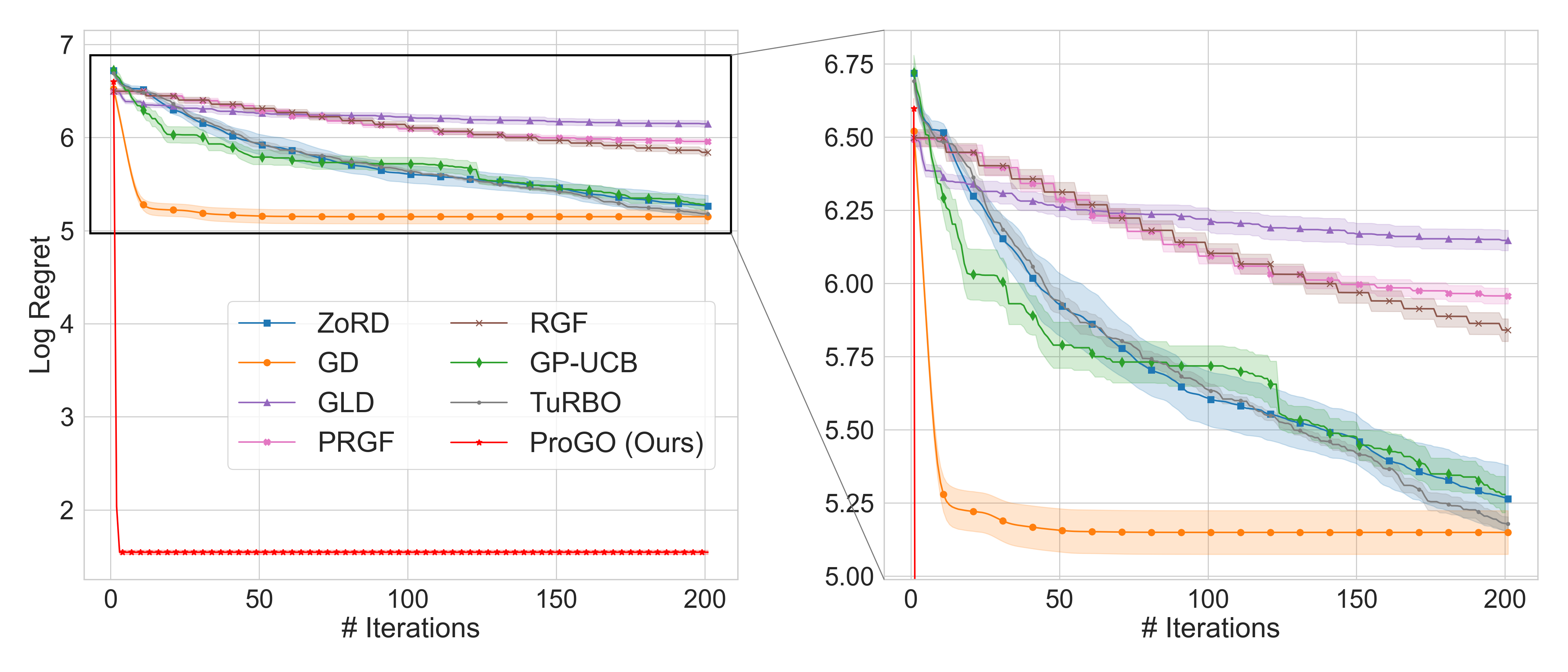}}
	\end{minipage}
	\caption{Evaluation of ProGO and competing methods applied to the Ackley function with $d=40, 100$. The $x$-axis represents the iteration count, while the $y$-axis denotes the average log-scaled regret. Each curve shows the mean $\pm$ standard error across ten independent runs.  \label{fig:levy100d}}
\end{figure}

\subsection{Levy}
\begin{figure}[H]
	\centering
        \begin{minipage}[t]{\linewidth}
		\centering
  \subcaptionbox{Surface of Levy ($d=2$)}{
  \includegraphics[width=.28\textwidth]{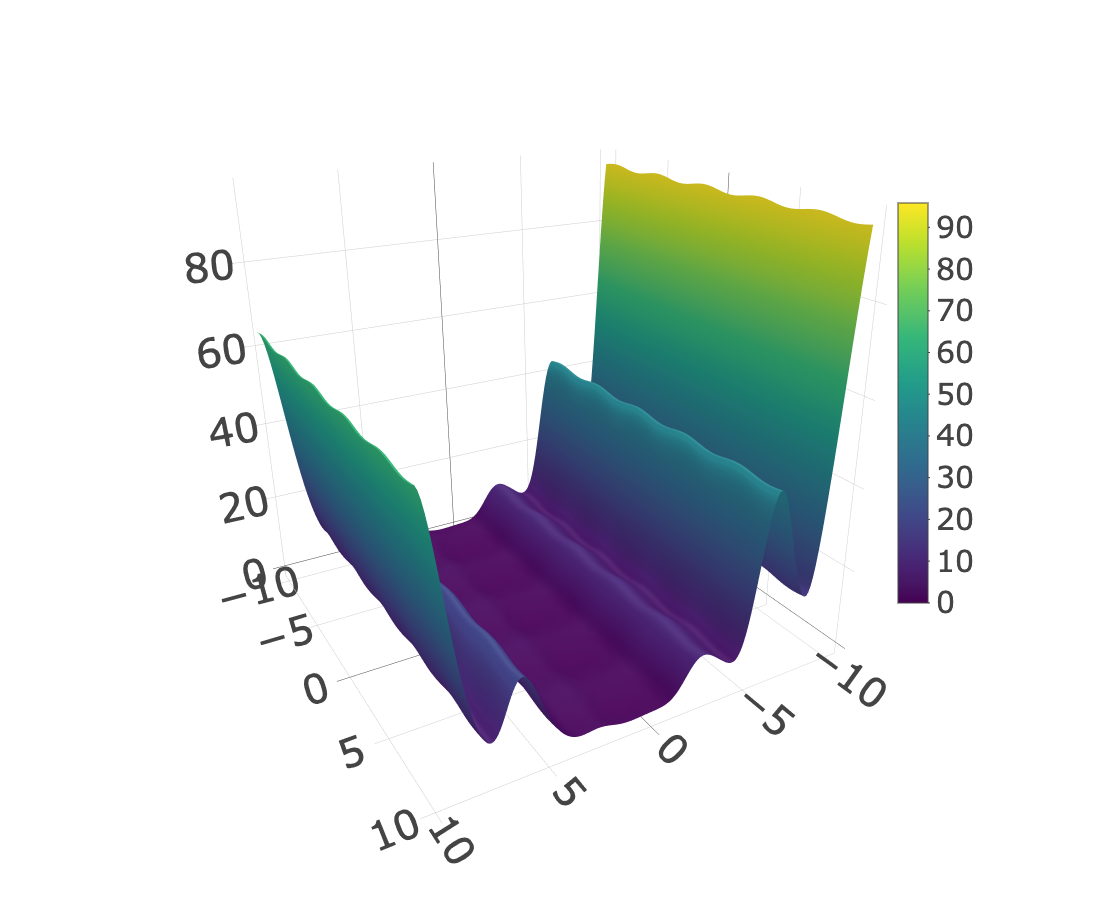}
  }
  \subcaptionbox{Results on Levy ($d=1000$)}{
  \includegraphics[width=.68\textwidth]{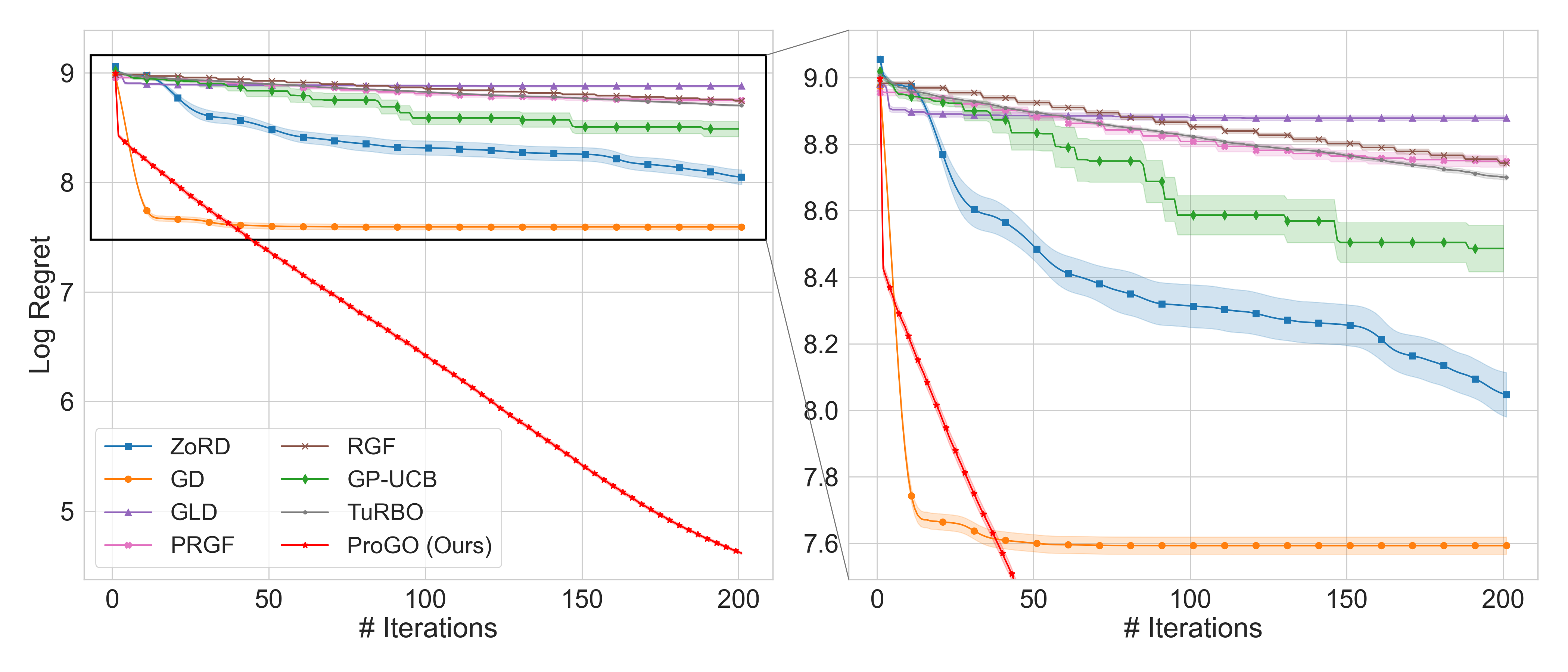}
  }
	\end{minipage}
	\caption{(a) Visualization of the Levy function for $d=2$. (b) Evaluation of ProGO and competing methods applied to the Levy function in a high-dimensional setting with $d=1000$. The $x$-axis represents the iteration count, while the $y$-axis denotes the average log-scaled regret. Each curve shows the mean $\pm$ standard error across ten independent runs. \label{fig:levy1000d} }
\end{figure}

The Levy function is another frequently used test function in optimization research, as in the work of \citet{shu2022zeroth}. It is a continuous and non-convex function defined as:
\begin{equation} \label{eq:levy}
f(\vx)=\sin ^2\left(\pi w_1\right)+\sum_{i=1}^{d-1}\left(w_i-1\right)^2\left[1+10 \sin ^2\left(\pi w_i+1\right)\right]+\left(w_d-1\right)^2\left[1+\sin ^2\left(2 \pi w_d\right)\right], 
\end{equation}
where  $w_i=1+ (\vx_i-1)/4,  \forall i=1, \ldots, d$. The global minimum is $f^*=0$ attained at ${\vx}^*=(1, \ldots, 1)^\top$. As shown in Figure \ref{fig:levy1000d}, the Levy function presents a more complex optimization challenge than the Ackley function due to the substantially flatter area that surrounds its global optima.

\begin{table}[H]
\begin{center}
\caption{Comparison of performance on the Levy function for dimensions $d=40, 100, 1000$: Each optimization method undergoes ten independent runs. Accuracy is quantified using the average function log regret ($r_f$) and the average minima log regret ($r_m$) across these ten runs. Computational efficiency is represented by the average runtime ($t$) across all ten runs, measured in seconds. \label{tab:levyexp}}

    \resizebox{\linewidth}{!}{  
    \begin{tabular}{@{}c rrr c rrr c rrr@{}} \toprule
        & \multicolumn{3}{c}{$d=40$} & & \multicolumn{3}{c}{$d=100$} &  & \multicolumn{3}{c}{$d=1000$} \\ \cmidrule{2-4} \cmidrule{6-8} \cmidrule {10-12} 
         \multicolumn{1}{c}{ Method}  &\multicolumn{1}{c}{\bf ${r_f}$} &\multicolumn{1}{c}{\bf ${r_m}$} 
        &\multicolumn{1}{c}{\bf ${t~ (\mathrm{s})}$}  & &\multicolumn{1}{c}{\bf ${r_f}$} &\multicolumn{1}{c}{\bf ${r_m}$} 
        &\multicolumn{1}{c}{\bf ${t~ (\mathrm{s})}$}  & &\multicolumn{1}{c}{\bf ${r_f}$} &\multicolumn{1}{c}{${r_m}$} 
        &\multicolumn{1}{c}{\bf ${t~ (\mathrm{s})}$}  \\ \midrule
        ZoRD &            4.08 &          1.56 & 585.69 & & 5.26 &          1.56 & 1514.97 & &            8.05 &          1.58 &  272.29 \\
        GD &            4.10 &          1.58 &    0.16 & &            5.15 &          1.61 &     0.15 & &            7.59 &          1.59 &     0.14\\
        GLD &            4.89 &          1.59 &    0.40 & &            6.15 &          1.60 &     0.39 & &            8.88 &          1.59 &     0.39\\
        PRGF &            4.68 &          1.59 &    0.08 & &           5.96 &          1.61 &     0.07 &  &            8.75 &          1.60 &     0.07 \\
        RGF &            4.59 &          1.58 &    0.08 & &           5.84 &          1.62 &     0.08 & &            8.74 &          1.59 &     0.09\\
        GP-UCB &            3.80 &          1.58 & 216.83 & &           5.28 &          1.58 &  651.38 & &             8.49 &          1.58 & 1400.41\\
        TuRBO &            3.58 &          1.56 &  77.66 & &           5.18 &          1.55 &  197.78 & &            8.70 &          1.59 &  265.71 \\
        \textbf{ProGO} &           \textbf{-0.05} &         \textbf{-0.96} &   11.85 & &            \textbf{1.55} &         \textbf{-0.50} &    27.33 & &            \textbf{4.62} &         \textbf{-0.01} &    311.99\\   \bottomrule
    \end{tabular}
    }
\end{center}
\end{table}

Empirical evaluation of the Levy function for dimensions $d=40, 100, 1000$ is presented in Table \ref{tab:levyexp}. Notably, ProGO demonstrates significantly lower log regrets relative to other competing methods across all dimensions. Furthermore, Figure \ref{fig:levy1000d} and Figure \ref{fig:levy100d} show that ProGO exhibits a markedly faster convergence rate compared to competing gradient-based, zeroth-order, and Bayesian optimization methods. Notice that Gradient Descent, depicted in orange, initially exhibits rapid convergence but is then trapped in local optima.

\section{Conclusion} \label{sec:conc}
In this paper, we introduce ProGO, a novel probabilistic global optimization algorithm leveraging minima distribution theory and latent slice sampling technique. Our methodology represents a significant departure from conventional gradient-based methods, offering robust convergence guarantees for global optima while preserving computational efficiency without using gradient information. 

Specifically, our contributions are threefold: We extend \citet{luo2018minima}'s theoretical framework to non-compact sets and prove its global convergence. Based on the generalized framework, we implement the ProGO algorithm, integrating a latent slice sampler for enhanced computational efficiency, especially for high dimensions. Finally, comprehensive experiments demonstrate ProGO's outstanding performance over state-of-the-art methods in terms of accuracy and convergence speed on various functions and dimensions.

However, it is worth noting that ProGO may not be suitable for optimization problems where function evaluation is computationally expensive. Future investigations on enhancing the algorithm's computational efficiency and extending the applicability of ProGO to more diverse problem domains may contribute to the growing field of global optimization.

\bibliographystyle{agsm} 
\bibliography{progo}

@article{luo2018minima,
  title={Minima distribution for global optimization},
  author={Luo, Xiaopeng},
  journal={arXiv preprint arXiv:1812.03457},
  year={2018}
}

@article{li2023latent,
  title={A latent slice sampling algorithm},
  author={Li, Yanxin and Walker, Stephen G},
  journal={Computational Statistics \& Data Analysis},
  volume={179},
  pages={107652},
  year={2023},
  publisher={Elsevier}
}

@article{robbins1951stochastic,
  title={A stochastic approximation method},
  author={Robbins, Herbert and Monro, Sutton},
  journal={The annals of mathematical statistics},
  pages={400--407},
  year={1951},
  publisher={JSTOR}
}

@article{liu2020primer,
  title={A primer on zeroth-order optimization in signal processing and machine learning: Principals, recent advances, and applications},
  author={Liu, Sijia and Chen, Pin-Yu and Kailkhura, Bhavya and Zhang, Gaoyuan and Hero III, Alfred O and Varshney, Pramod K},
  journal={IEEE Signal Processing Magazine},
  volume={37},
  number={5},
  pages={43--54},
  year={2020},
  publisher={IEEE}
}

@article{golovin2019gradientless,
  title={Gradientless descent: High-dimensional zeroth-order optimization},
  author={Golovin, Daniel and Karro, John and Kochanski, Greg and Lee, Chansoo and Song, Xingyou and Zhang, Qiuyi},
  journal={arXiv preprint arXiv:1911.06317},
  year={2019}
}

@article{cheng2021convergence,
  title={On the convergence of prior-guided zeroth-order optimization algorithms},
  author={Cheng, Shuyu and Wu, Guoqiang and Zhu, Jun},
  journal={Advances in Neural Information Processing Systems},
  volume={34},
  pages={14620--14631},
  year={2021}
}

@article{nesterov2017random,
  title={Random gradient-free minimization of convex functions},
  author={Nesterov, Yurii and Spokoiny, Vladimir},
  journal={Foundations of Computational Mathematics},
  volume={17},
  pages={527--566},
  year={2017},
  publisher={Springer}
}

@article{rios2013derivative,
  title={Derivative-free optimization: a review of algorithms and comparison of software implementations},
  author={Rios, Luis Miguel and Sahinidis, Nikolaos V},
  journal={Journal of Global Optimization},
  volume={56},
  pages={1247--1293},
  year={2013},
  publisher={Springer}
}

@article{larson2019derivative,
  title={Derivative-free optimization methods},
  author={Larson, Jeffrey and Menickelly, Matt and Wild, Stefan M},
  journal={Acta Numerica},
  volume={28},
  pages={287--404},
  year={2019},
  publisher={Cambridge University Press}
}

@article{jamil2013literature,
  title={A literature survey of benchmark functions for global optimisation problems},
  author={Jamil, Momin and Yang, Xin-She},
  journal={International Journal of Mathematical Modelling and Numerical Optimisation},
  volume={4},
  number={2},
  pages={150--194},
  year={2013},
  publisher={Inderscience Publishers Ltd}
}

@book{ackley2012connectionist,
  title={A connectionist machine for genetic hillclimbing},
  author={Ackley, David},
  volume={28},
  year={2012},
  publisher={Springer science \& business media}
}

@article{neal2003slice,
  title={Slice sampling},
  author={Neal, Radford M},
  journal={The annals of statistics},
  volume={31},
  number={3},
  pages={705--767},
  year={2003},
  publisher={Institute of Mathematical Statistics}
}

@article{corriou2021analytical,
  title={Analytical Methods for Optimization},
  author={Corriou, Jean-Pierre and Corriou, Jean-Pierre},
  journal={Numerical methods and optimization: Theory and practice for engineers},
  pages={455--503},
  year={2021},
  publisher={Springer}
}

@article{reddi2019convergence,
  title={On the convergence of adam and beyond},
  author={Reddi, Sashank J and Kale, Satyen and Kumar, Sanjiv},
  journal={arXiv preprint arXiv:1904.09237},
  year={2019}
}

@article{duchi2011adaptive,
  title={Adaptive subgradient methods for online learning and stochastic optimization.},
  author={Duchi, John and Hazan, Elad and Singer, Yoram},
  journal={Journal of machine learning research},
  volume={12},
  number={7},
  year={2011}
}

@inproceedings{mnih2016asynchronous,
  title={Asynchronous methods for deep reinforcement learning},
  author={Mnih, Volodymyr and Badia, Adria Puigdomenech and Mirza, Mehdi and Graves, Alex and Lillicrap, Timothy and Harley, Tim and Silver, David and Kavukcuoglu, Koray},
  booktitle={International conference on machine learning},
  pages={1928--1937},
  year={2016},
  organization={PMLR}
}

@inproceedings{seward2018first,
  title={First order generative adversarial networks},
  author={Seward, Calvin and Unterthiner, Thomas and Bergmann, Urs and Jetchev, Nikolay and Hochreiter, Sepp},
  booktitle={International Conference on Machine Learning},
  pages={4567--4576},
  year={2018},
  organization={PMLR}
}

@article{tieleman2012rmsprop,
  title={Rmsprop: Divide the gradient by a running average of its recent magnitude. coursera: Neural networks for machine learning},
  author={Tieleman, Tijmen and Hinton, Geoffrey},
  journal={COURSERA Neural Networks Mach. Learn},
  volume={17},
  year={2012}
}

@inproceedings{shi2020rmsprop,
  title={RMSprop converges with proper hyper-parameter},
  author={Shi, Naichen and Li, Dawei and Hong, Mingyi and Sun, Ruoyu},
  booktitle={International Conference on Learning Representations},
  year={2020}
}

@article{adorio2005mvf,
  title={Mvf-multivariate test functions library in c for unconstrained global optimization},
  author={Adorio, Ernesto P and Diliman, U},
  journal={Quezon City, Metro Manila, Philippines},
  volume={44},
  year={2005}
}

@article{mira2002efficiency,
  title={Efficiency and convergence properties of slice samplers},
  author={Mira, Antonietta and Tierney, Luke},
  journal={Scandinavian Journal of Statistics},
  volume={29},
  number={1},
  pages={1--12},
  year={2002},
  publisher={Wiley Online Library}
}

@inproceedings{shu2022zeroth,
  title={Zeroth-Order Optimization with Trajectory-Informed Derivative Estimation},
  author={Shu, Yao and Dai, Zhongxiang and Sng, Weicong and Verma, Arun and Jaillet, Patrick and Low, Bryan Kian Hsiang},
  booktitle={The Eleventh International Conference on Learning Representations},
  year={2022}
}

@article{srinivas2009gaussian,
  title={Gaussian process optimization in the bandit setting: No regret and experimental design},
  author={Srinivas, Niranjan and Krause, Andreas and Kakade, Sham M and Seeger, Matthias},
  journal={arXiv preprint arXiv:0912.3995},
  year={2009}
}

@InProceedings{KingBa15,
  author    = {Kingma, Diederik and Ba, Jimmy},
  booktitle = {International Conference on Learning Representations (ICLR)},
  title     = {Adam: A Method for Stochastic Optimization},
  year      = {2015},
  address   = {San Diega, CA, USA},
  optmonth  = {12},
}

@article{eriksson2019scalable,
  title={Scalable global optimization via local Bayesian optimization},
  author={Eriksson, David and Pearce, Michael and Gardner, Jacob and Turner, Ryan D and Poloczek, Matthias},
  journal={Advances in neural information processing systems},
  volume={32},
  year={2019}
}

@article{snoek2012practical,
  title={Practical bayesian optimization of machine learning algorithms},
  author={Snoek, Jasper and Larochelle, Hugo and Adams, Ryan P},
  journal={Advances in neural information processing systems},
  volume={25},
  year={2012}
}

@inproceedings{harrison2010introduction,
  title={Introduction to monte carlo simulation},
  author={Harrison, Robert L},
  booktitle={AIP conference proceedings},
  volume={1204},
  number={1},
  pages={17--21},
  year={2010},
  organization={American Institute of Physics}
}

@inproceedings{murray2010elliptical,
  title={Elliptical slice sampling},
  author={Murray, Iain and Adams, Ryan and MacKay, David},
  booktitle={Proceedings of the thirteenth international conference on artificial intelligence and statistics},
  pages={541--548},
  year={2010},
  organization={JMLR Workshop and Conference Proceedings}
}

@article{rudolf2023dimension,
  title={Dimension-independent spectral gap of polar slice sampling},
  author={Rudolf, Daniel and Sch{\"a}r, Philip},
  journal={arXiv preprint arXiv:2305.03685},
  year={2023}
}

@inproceedings{ru2019bayesopt,
  title={Bayesopt adversarial attack},
  author={Ru, Binxin and Cobb, Adam and Blaas, Arno and Gal, Yarin},
  booktitle={International Conference on Learning Representations},
  year={2019}
}

@article{roberts1999convergence,
  title={Convergence of slice sampler Markov chains},
  author={Roberts, Gareth O and Rosenthal, Jeffrey S},
  journal={Journal of the Royal Statistical Society Series B: Statistical Methodology},
  volume={61},
  number={3},
  pages={643--660},
  year={1999},
  publisher={Oxford University Press}
}

\appendix
\section{Proofs} \label{sec:proof}
\subsection{Proof for Theorem \ref{thm:1}} \label{sec:proof1}
\begin{proof}[Proof for Theorem \ref{thm:1}]

Recall that $\Omega^* = \{\vx \in \Omega: f(\vx) = 0\}$ without loss of generalizability. For any $\epsilon > 0$, define $\Omega_{\epsilon} = \{ \vx \in \Omega: f(\vx) < \epsilon\}$ and $\Omega_{\epsilon}^c = \Omega \setminus \Omega_{\epsilon}$, then

\begin{align*}
    \int_{\Omega} f(\vx) m_k(\vx) d\vx &= \frac{ \int_{\Omega_{\frac{\epsilon}{2}}} f(\vx) \cdot e^{-kf(\vx)} \pi(\vx) \mathrm{d}\vx +  \int_{\Omega_{\frac{\epsilon}{2}}^c} f(\vx) \cdot  e^{-kf(\vx)} \cdot \pi(\vx) \mathrm{d}\vx}{  \int_{\Omega} e^{-kf(t)} \cdot \pi(t) \mathrm{d}t} \\
    & < \frac{\epsilon}{2} \cdot \frac{ \int_{\Omega_{\frac{\epsilon}{2}}} e^{-kf(\vx)} \cdot \pi(\vx) \mathrm{d}\vx}{  \int_{\Omega} e^{-kf(t)} \cdot \pi(t) \mathrm{d}t} + \frac{ \int_{\Omega_{\frac{\epsilon}{2}}^c} f(\vx) \cdot e^{-kf(\vx)} \cdot \pi(\vx) \mathrm{d}\vx}{  \int_{\Omega_{\frac{\epsilon}{4}}} e^{-kf(t)} \cdot \pi(t) \mathrm{d}t} \\
    & \leqslant \frac{\epsilon}{2} \cdot 1 + \frac{ \int_{\Omega_{\frac{\epsilon}{2}}^c} f(\vx) e^{-k \left(f(\vx) - \frac{\epsilon}{4} \right)} \cdot \pi(\vx) \mathrm{d}\vx}{  \int_{\Omega_{\frac{\epsilon}{4}}}  \pi(t) \mathrm{d}t} \\
    & = \frac{\epsilon}{2} + \int_{\Omega_{\frac{\epsilon}{2}}^c} \frac{f(\vx)\pi(\vx) e^{-k \left(f(\vx) - \frac{\epsilon}{4} \right)}}{\int_{\Omega_{\frac{\epsilon}{4}}}  \pi(t) \mathrm{d}t}   \mathrm{d}\vx \stackrel{\Delta}{=} \frac{\epsilon}{2} + \int_{\Omega_{\frac{\epsilon}{2}}^c} g_k(\vx) \mathrm{d}\vx.
\end{align*}

The first inequality lies in $f(\vx) < \frac{\epsilon}{2}$ for any $\vx \in \Omega_{\frac{\epsilon}{2}}$ and $1/{  \int_{\Omega} e^{-kf(t)} \cdot \pi(t) \mathrm{d}t} < 1/{  \int_{\Omega_{\frac{\epsilon}{2}}} e^{-kf(t)} \cdot \pi(t) \mathrm{d}t}$. The second inequality is given by ${ \int_{\Omega_{\frac{\epsilon}{2}}} e^{-kf(\vx)} \cdot \pi(\vx) \mathrm{d}\vx} < {  \int_{\Omega} e^{-kf(\vx)} \cdot \pi(\vx) \mathrm{d}\vx}$ and $f(\vx) < \frac{\epsilon}{4}$ for any $\vx \in \Omega_{\frac{\epsilon}{4}}$. Define $g_k(\vx) = \frac{f(\vx)\pi(\vx) e^{-k \left(f(\vx) - \frac{\epsilon}{4} \right)}}{\int_{\Omega_{\frac{\epsilon}{4}}}  \pi(t) \mathrm{d}t}$ for $\vx \in {\Omega_{\frac{\epsilon}{2}}^c}$, then $\{ g_k(\vx)\}_{k=1}^{\infty}$ is a sequence of nonnegative functions that monotonously decreases to 0 when k goes to infinity and converges to zero, i.e., $\lim_{k \rightarrow \infty}g_k(\vx) = 0$. Hence, by monotone convergence theorem, $
\lim_{k \rightarrow \infty} { \int_{\Omega_{\frac{\epsilon}{2}}^c} g_k(\vx) \mathrm{d}\vx} = 0.
$
Consequently, for any $\epsilon > 0$, there exists a large $K_0$, such that $\int_{\Omega_{\frac{\epsilon}{2}}^c} g_k(\vx) \mathrm{d}\vx < \epsilon/2$ for any $k>K_0$, and thus 
\begin{align*}
    0  \leqslant \int_{\Omega} f(\vx) m_k(\vx) d\vx &< \frac{\epsilon}{2} + \frac{\epsilon}{2} = \epsilon,
\end{align*}
which proves $
\lim _{k \rightarrow \infty} \int_{\Omega} f(\vx) m_k(\vx) \mathrm{d} \vx=f^* (=0).$

\end{proof}

\subsection{Proof for Theorem \ref{thm:2}} \label{sec:proof2}
\begin{proof}[Proof for Theorem \ref{thm:2}]

For every $k \in \mathbb{R}$,
\begin{align*}
    \frac{\dd}{\dd k} \log m_k(\vx)  &= \frac{\dd}{\dd k} \left( -k f(\vx) + \log \pi(\vx) - \log \left\{\int_{\Omega} e^{-kf(\vt) } \cdot \pi(\vt) \dd \vt \right\} \right) \\
    & \stackrel{}{=}  -f(\vx) - \frac{  \frac{\dd}{\dd k} \left\{ \int_{\Omega} e^{-kf(\vt) } \cdot \pi(\vt) \dd \vt   \right\}    }{  \int_{\Omega} e^{-kf(\vt) } \cdot \pi(\vt) \dd \vt   }  \\ 
    & \stackrel{}{=}  -f(\vx) - \frac{  \int_{\Omega} (-f(\vt)) \cdot  e^{-kf(\vt) } \cdot \pi(\vt) d\vt    }{  \int_{\Omega} e^{-kf(\vt) } \cdot \pi(\vt) \dd \vt   }   \\ 
    &= \mathbb{E}_{m_k}(f) - f(\vx).
\end{align*}
Hence,
\begin{align*}
    \frac{\dd}{\dd k}m_k(\vx) = m_k(\vx) \cdot  \frac{\dd}{\dd k} \log m_k(\vx) = m_k(\vx) \cdot \left( \mathbb{E}_{m_k}(f) - f(\vx) \right).
\end{align*}
Then we have
\begin{align*}
    m^{(k+\Delta k)} (\vx) &= m_k(\vx) + \int_{k}^{k+\Delta k} \frac{\dd}{\dd \vv}m^{(\vv)}(\vx) \dd \vv \\
    & = m_k(\vx) + \int_{k}^{k+\Delta k} m^{(\vv)}(\vx) \cdot \left[\mathbb{E}^{(\vv)}(f) - f(\vx) \right] \dd \vv. 
\end{align*}
Then there exist a $\xi \in (k, k+\Delta k)$ such that
\begin{align*}
    \frac{\mathbb{E}^{(k+\Delta k)}(f)-\mathbb{E}_{m_k}(f)}{\Delta k} & =\frac{1}{\Delta k} \int_{\Omega} f(\vx)\left(m^{(k+\Delta k)}(\vx)-m_k(\vx)\right) \mathrm{d} \vx \\
    & =\frac{1}{\Delta k} \int_{\Omega} \int_k^{k+\Delta k} f(\vx) m^{(\vv)}(\vx) \left[\mathbb{E}^{(\vv)}(f) - f(\vx) \right] \mathrm{d} \vv \mathrm{~d} \vx \\
    & \stackrel{(R \ref{rmk:3})}{=} \frac{1}{\Delta k} \int_k^{k+\Delta k} \int_{\Omega}  f(\vx) m^{(\vv)}(\vx) \left[\mathbb{E}^{(\vv)}(f) - f(\vx) \right] \mathrm{d} \vx \mathrm{~d} \vv  \\
    & =\frac{1}{\Delta k} \int_k^{k+\Delta k}\left\{ \left[ \mathbb{E}^{(\vv)}(f) \right]^2 - \mathbb{E}^{(\vv)}(f^2) \right\} \mathrm{d} \vv \\
    & = \left[\mathbb{E}^{(\xi)}(f) \right]^2 - \mathbb{E}^{(\xi)}(f^2),
\end{align*}

where the exchangeability of the integral is proved by Fubini Theorem in Remark \ref{rmk:3}. Hence, we have
\begin{align*}
\frac{\mathrm{d} \mathbb{E}_{m_k}(f)}{\mathrm{d} k} & =\lim _{\Delta k \rightarrow 0} \frac{\mathbb{E}^{(k+\Delta k)}(f)-\mathbb{E}_{m_k}(f)}{\Delta k} \\
& = \left\{ \mathbb{E}_{m_k}(f) \right \}^2 - \mathbb{E}_{(k)}(f^2) \\
& = - \int_{\Omega} \left( f(\vx) -  \mathbb{E}_{m_k}(f)\right)^2 m_k (\vx) \mathrm{d} \vx \\
&= - \mathbb{V}{ar}^{(k)}(f) \leqslant 0,
\end{align*}
where the equality holds only when $\mathbb{V}{ar}^{(k)}(f) =0$, i.e., $f$ is a constant function on $\Omega$.

\begin{remark}[R3] \label{rmk:3} 
    Define $h_k(t) = te^{-kt}$ for any $t \in \mathbb{R}$, then its first derivative is $h_k^{\prime}(t)=(1-kt)e^{-kt}$. For any $t < \frac{1}{k}$, $h_k^{\prime}(t)>0$, and thus $g(t)$ is increasing when $t< \frac{1}{k}$; for any $t \geqslant \frac{1}{k}$, $h_k^{\prime}(t) \leqslant 0$, and thus $g(t)$ is non-increasing when $t \geqslant \frac{1}{k}$. Therefore, $h_k(t) \leqslant h_k(t=\frac{1}{k})=\frac{1}{ke}$. Define $g_k(t)=t^2 e^{-kt}$, then the absolute value of $g_k(t)$ has its upper bound as $\frac{4}{k^2e^2}$ using similar strategy.
    For any $k > 0$, the denominator of $\mathbb{E}^{(\vv)}(f) $ is a finite positive constant, where $0 < \int_{\Omega} e^{-k f(\vx)} \pi (\vx) \mathrm{d} \vx \stackrel{\Delta}{=} \alpha_k \leqslant e^{-k f^*}   \int_{\Omega} \pi (\vx) \mathrm{d} \vx \leqslant e^{-kf^*}$, and thus $\mathbb{E}^{(\vv)}(f)$ defined in the proof of Theorem \ref{thm:2} is bounded following:
\begin{equation}
    \begin{aligned}
    \mathbb{E}^{(\vv)}(f) &=  \frac{\int_{\Omega} f(\vx) e^{-\vv f(\vx)} \pi (\vx) \mathrm{d} \vx}{\int_{\Omega} e^{-\vv f(\vx)} \pi (\vx) \mathrm{d} \vx} \\
    & = \frac{\int_{\Omega} h_v(f(\vx)) \pi (\vx) \mathrm{d} \vx}{\int_{\Omega} e^{-\vv f(\vx)} \pi (\vx) \mathrm{d} \vx} \\
    & \leqslant \frac{ \frac{1}{ve} \int_{\Omega} \pi (\vx) \mathrm{d} \vx}{\int_{\Omega} e^{-\vv f(\vx)} \pi (\vx) \mathrm{d} \vx} = \frac{1}{ve\alpha_v}
\end{aligned}
\end{equation}

Similarly, $\mathbb{E}^{(\vv)}(f^2)$ is also bounded by:
\begin{equation}
    \begin{aligned}
    \mathbb{E}^{(\vv)}(f^2) &=  \frac{\int_{\Omega} f^2(\vx) e^{-\vv f(\vx)} \pi (\vx) \mathrm{d} \vx}{\int_{\Omega} e^{-\vv f(\vx)} \pi (\vx) \mathrm{d} \vx} \\
    & = \frac{\int_{\Omega} g_v(f(\vx)) \pi (\vx) \mathrm{d} \vx}{\int_{\Omega} e^{-\vv f(\vx)} \pi (\vx) \mathrm{d} \vx} \\
    & \leqslant \frac{ \frac{4}{\vv^2e^2} \int_{\Omega} \pi (\vx) \mathrm{d} \vx}{\int_{\Omega} e^{-\vv f(\vx)} \pi (\vx) \mathrm{d} \vx} = \frac{4}{\vv^2e^2\alpha_v}
\end{aligned}
\end{equation}

Hence, given bounded $\mathbb{E}^{(\vv)}(f)$ and bounded $\mathbb{E}^{(\vv)}(f)$,  the following order of integration is exchangeable by the Fubini theorem:

$$
\int_{\Omega} \int_k^{k+\Delta k} \mathbb{E}^{(\vv)}(f) \mathrm{d} \vv \mathrm{~d} \vx =  \int_k^{k+\Delta k} \int_{\Omega}  \mathbb{E}^{(\vv)}(f) \mathrm{d} \vx \mathrm{~d} \vv, 
$$
$$
\int_{\Omega} \int_k^{k+\Delta k} \mathbb{E}^{(\vv)}(f^2) \mathrm{d} \vv \mathrm{~d} \vx =  \int_k^{k+\Delta k} \int_{\Omega}  \mathbb{E}^{(\vv)}(f^2) \mathrm{d} \vx \mathrm{~d} \vv.
$$
\end{remark}
\end{proof}

\subsection{Proof for Theroem \ref{thm:3}} \label{sec:proof3}

\begin{proof}[Proof for Theroem \ref{thm:3}]
First, we prove eq. \ref{eq:thm31}. 

For $\forall k > 0$, denote $c_k=\log \left\{ \int_{\Omega} e^{-kf(\vt)} \pi(\vt) \dd\vt\right\} $, then 
\begin{align*}
    \operatorname*{argmax}_{\vx \in \Omega} m_k (\vx) &= \operatorname*{argmax}_{\vx \in \Omega} \log m_k (\vx) \\
    & = \operatorname*{argmax}_{\vx \in \Omega} \left( -k f(\vx) +  \log \pi(\vx) -c_k \right) \\
    & = \operatorname*{argmin}_{\vx \in \Omega} \left( f(\vx) - \frac{1}{k} \log \pi(\vx) \right),
\end{align*}
which indicates the maximizers of $m_k(\vx)$ are the minimizers of $\left( f(\vx) - \frac{1}{k} \log \pi(\vx) \right)$. Assume the upper bound for $\pi(\vx)$ is $B$, then for any $\vx_k^* \in \Omega_k$ and $\tilde{\vx} \in \Omega_*$, we have
\begin{equation} \label{eq:1}
    f(\vx_k^*) - \frac{1}{k}\log B  \leqslant f(\vx_k^*)- \frac{1}{k} \log \pi(\vx_k^*) \leqslant f(\tilde{\vx}) - \frac{1}{k} \log \pi(\tilde{\vx}),
\end{equation}

where the first inequality of eq. (\ref{eq:1}) holds for the bounded density, where $\pi(\vx) \leqslant B, \forall \vx \in \Omega$. The second inequality of eq. (\ref{eq:1}) lies in $\vx_k^* \in \arg \max m_k(\vx)$. Furtherly,  $\lim_{k \rightarrow \infty } f(\vx_k^*) =f^*$ follows from
\begin{equation} \label{eq:2}
    f^* \leqslant \liminf_{k \rightarrow \infty} f(\vx_k^*) 
 \leqslant \limsup_{k \rightarrow \infty} f(\vx_k^*) \leqslant \limsup_{k \rightarrow \infty} \left\{ f(\tilde{\vx}) - \frac{1}{k} \log \pi(\tilde{\vx}) \right \} =f^*,
\end{equation}
where the three inequalities follow from the definition of $f^*$ ($\leqslant f(\vx)$, for $\forall \vx \in \Omega$), the definition of limit inferior and limit superior, and the limit superior of eq. (\ref{eq:1}) respectively.

Next, we prove eq. (\ref{eq:thm32}). In the previous part, we have already proved $\lim_{k \rightarrow \infty} f(\vx_k^*) =f^*,$ i.e., for any $\epsilon >0$, there exists large $K_0>0$, such that for $\forall k > K_0$, it holds that
\begin{equation} \label{eq:3}
    f(\vx_k^*) - f^* < \epsilon.
\end{equation}

Suppose there exists a large $K>K_0$, such that 
$$
\inf_{ \tilde{\vx} \in \Omega^*} ||\vx_K^* - \tilde{\vx}|| > \delta, \quad \text{for } \forall \tilde{\vx} \in \Omega^*.
$$
Then $f(\vx_K^*) - f^* > \epsilon$ given the strong separability condition in Assumption \ref{assumption2}, which is contradictory to eq. (\ref{eq:3}). Thus, we have completed the proof for Theorem \ref{thm:3} using the proof by contradictory technique.
\end{proof}

\begin{theorem} The nascent minima distribution function defined in Definition \ref{def:1} satisfies:
\label{thm:4}
\begin{enumerate}
    \item[(i)] For $\forall k \in \mathbb{R}$, $m_k(\vx)$ is a PDF on $\Omega$; especially when $k=0$, $m^{(0)}(\vx)=\pi(\vx)/ \Pi(\Omega)$, where $\Pi(S)= \int_{S} \pi(\vx)\mathrm{d} \vx$ is the probability measure of $S \subseteq \mathbb{R}^d$.
    \item[(ii)] If $\nabla \pi$ and $\nabla f: \mathbb{R}^d \rightarrow \mathbb{R}^d$ are continuous real functions on each dimension, then 
    $$\nabla \log m_k(x) = -k \nabla f(x) + \frac{\nabla \pi(x)}{\pi(x)}.$$
    \item[(ii)] For every $k \in \mathbb{R}$, it holds that
    \begin{align*}
        \frac{\dd}{\dd k} \log m_k (\vx) = \mathbb{E}_{m_k}(f) - f(\vx),
    \end{align*}
    where $\mathbb{E}_{m_k}(f)=\int_{\Omega}f(\vx)m_k(\vx)\mathrm{d} \vx$.
    \item[(iii)] For any $\vx \in \Omega$, consider the sequence $m_k(\vx)$, where $k > 0$ and is going to infinity:
    \begin{itemize}
        \item [(a)] If $\Omega^*$ has zero probability measure, i.e., $\Pi(\Omega^*) \triangleq \int_{\Omega^*}\pi(\vx)\mathrm{d} \vx=0$, then
    \begin{align*}
       m_{\infty}(\vx)  =  \lim_{k \rightarrow \infty}m_k(\vx) = \left\{ 
        \begin{array}{cc}
            0, & \vx \notin \Omega^*;  \\
            \infty, & \vx \in \Omega^*.
        \end{array}
        \right.
    \end{align*}
    \item[(b)] If $\Omega^*$ has nonzero probability measure, i.e., $\Pi(\Omega^*)>0$, then
    \begin{align*}
        \lim_{k \rightarrow \infty}m_k(\vx) = \left\{ 
        \begin{array}{cc}
            0, & \vx \notin \Omega^*;  \\
            \frac{\pi(\vx)}{\Pi(\Omega^*)}, & \vx \in \Omega^*.
        \end{array}
        \right.
    \end{align*}
    \end{itemize}

\item[{(iv)}] Define $\Omega_k$ as the set of maximizers of $m_k(\vx)$ on $\vx \in \Omega$, if $\pi(\vx)$ is bounded, then for any $k > 0$,  the sequence of $\vx_k^* \in \Omega_k$ satisfies the following:
    \begin{enumerate}
        \item $\lim_{k \rightarrow \infty } f(\vx_k^*) =0 (=f^*)$.
        \item $\pi(\vx_k^*) \geqslant \pi^* \stackrel{\Delta}{=} \max_{\vx^* \in \Omega^*} \pi(\vx^*)$.
        \item The sequence $\{\pi(\vx_k^*)\}_{k=1,2,\cdots}$ is non-increasing and converges to a limit, where $ \lim_{k \rightarrow \infty} \pi(\vx_k^*) = \liminf_{k \rightarrow \infty} \pi(\vx_k^*) \stackrel{\Delta}{=} \pi_0$.
    \end{enumerate}
\end{enumerate}
\end{theorem}

\begin{proof} Clearly, (i) follows from the definition of $m_k(\vx)$ in Definition \ref{def:1}.

\begin{itemize}
\item For every $k \in \mathbb{R}$, (ii) follows from
\begin{align*}
    \frac{\dd}{\dd k} \log m_k(\vx)  &= \frac{\dd}{\dd k} \left( -k f(\vx) + \log \pi(\vx) - \log \left\{\int_{\Omega} e^{-kf(t) } \cdot \pi(t) dt \right\} \right) \\
    & \stackrel{}{=}  -f(\vx) - \frac{  \frac{\dd}{\dd k} \left\{ \int_{\Omega} e^{-kf(t) } \cdot \pi(t) dt   \right\}    }{  \int_{\Omega} e^{-kf(s) } \cdot \pi(s) ds   }  \\ 
    & \stackrel{}{=}  -f(\vx) - \frac{  \int_{\Omega} (-f(t)) \cdot  e^{-kf(t) } \cdot \pi(t) dt    }{  \int_{\Omega} e^{-kf(s) } \cdot \pi(s) ds   }   \\ 
    &= \mathbb{E}_{m_k}(f) - f(\vx), 
\end{align*}

where $\frac{\dd}{\dd k} \left\{ \int_{\Omega} e^{-kf(t) } \cdot \pi(t) dt   \right\} =  \int_{\Omega} \frac{\dd}{\dd k} \left\{  e^{-kf(t) } \cdot \pi(t) \right\}dt$ follows from the Tonelli theorem for the exchangeable order of the integral and derivatives, given that the function $\left\{  e^{-kf(t) } \cdot \pi(t) \right\}$ is non-negative.

\item For (iii) notice that by Remark \ref{rmk:1}, 

\begin{equation*}\label{eq:3.1}
    \int_{\Omega}e^{-kf(\vx)} \pi(\vx) \mathrm{d} \vx = \int_{\Omega^*} \pi(\vx) \mathrm{d} \vx + \int_{\Omega \setminus \Omega^*} e^{-kf(\vx)}\pi(\vx) \mathrm{d} \vx,
\end{equation*}
since $f(\vx)=0$ for $\forall \vx \in \Omega^*$.
In addition, since $\left \{ e^{-kf(\vx)} \pi(\vx) \right\}$ is monotonely decreasing as $k$ increases and $\lim_{k \rightarrow \infty 0} \left\{ e^{-kf(\vx)} \pi(\vx) \right \} =0$ due to $f(\vx)>0$ for any $\vx \notin \Omega^*$, then
\begin{equation*} \label{eq:3.2}
    \lim_{k \rightarrow \infty} \int_{\Omega}e^{-kf(\vx)} \pi(\vx) \mathrm{d} \vx = \int_{\Omega^*} \pi(\vx) \mathrm{d} \vx + \lim_{k \rightarrow \infty} \int_{\Omega \setminus \Omega^*} e^{-kf(\vx)}\pi(\vx) \mathrm{d} \vx  = \Pi(\Omega^*),
\end{equation*}
which follows from the monotone convergence theorem. \begin{itemize}
    \item [(a)] Hence, for any $\vx \in \Omega^*$,
$$
m_{\infty}(\vx)  = \lim_{k \rightarrow \infty}  m_k(\vx) = \lim_{k \rightarrow \infty}  \frac{e^{-kf(\vx)} \cdot \pi(\vx) }{\int_{\Omega} e^{-kf(t) } \cdot \pi(t) dt } =  \left\{ 
\begin{array}{cc}
    \infty, & \text{if } \Pi(\Omega^*)=0; \\
    \frac{\pi(\vx)}{\Pi(\Omega^*)}, & \text{if } \Pi(\Omega^*) \neq 0.
\end{array}
\right.
$$
\item [(b)]  For any $\vx^{\prime} \notin \Omega^*$, it follows from the continuity of $f$ that there exists a set $\Omega_{\vx^{\prime}}$ such that $f(t) < f(\vx^{\prime})$ for any $t \in \Omega_{\vx^{\prime}}$, hence,
\end{itemize}

\begin{align*}
    m_k(\vx^{\prime}) &= \frac{e^{-k f(\vx^{\prime})} \cdot \pi(\vx^{\prime})}{ 
    \int_{\Omega_{\vx^{\prime}} } e^{-kf(t) } \cdot \pi(t) dt  + \int_{\Omega  \setminus \Omega_{\vx^{\prime}} } e^{-kf(t) } \cdot \pi(t) dt  } \\
    & \leqslant \frac{ \pi(\vx^{\prime})}{ \int_{\Omega_{\vx^{\prime}} } (e^{-f(t)}/ e^{-f(\vx^{\prime})} )^k \cdot \pi(t) dt     },
\end{align*}
since $e^{-f(t)}/ e^{-f(\vx^{\prime})} > 1$ for any $t \in \Omega_{\vx^{\prime}}$ and $\pi(t) > 0$, the limit of $\int_{\Omega_{\vx^{\prime}}} (e^{-f(t)}/e^{-f(\vx^{\prime})} )^k \cdot \pi(t) dt$ tends to $\infty$ as $k \rightarrow \infty$; thus, it holds that
\begin{align*}
    \operatorname*{lim}_{k\rightarrow \infty}m_k(\vx^{\prime})=0, \forall \vx^{\prime} \notin \Omega^*.
\end{align*}

\item  (iv) describes the properties of the maximizers of $m_k(\vx)$. For $\forall k > 0$, denote $c_k=\log \left\{ \int_{\Omega} e^{-kf(t)} \pi(t) dt\right\} $, then 
\begin{align*}
    \operatorname*{argmax}_{\vx \in \Omega} m_k (\vx) &= \operatorname*{argmax}_{\vx \in \Omega} \log m_k (\vx) \\
    & = \operatorname*{argmax}_{\vx \in \Omega} \left( -k f(\vx) +  \log \pi(\vx) -c_k \right) \\
    & = \operatorname*{argmin}_{\vx \in \Omega} \left( f(\vx) - \frac{1}{k} \log \pi(\vx) \right)  \\
    &= \operatorname*{argmin}_{\vx \in \Omega} g_k(\vx),
\end{align*}
which indicates the maximizers of $m_k(\vx)$ are the minimizers of $g_k(\vx)$, i.e., $\Omega_k = \{\vx^{\prime}: m_k(\vx^\prime) \geqslant m_k(\vx), \forall \vx \in \Omega\} = \{\vx^{\prime \prime}: g_k(\vx^{\prime \prime}) \leqslant g_k(\vx), \forall \vx \in \Omega\}$. 

\begin{enumerate}
    \item [(a)] For any $\vx_k^* \in \Omega_k$ and $\vx^* \in \Omega^*$, we have
$$
f(\vx_k^*) - \frac{1}{k}\log B  \leqslant f(\vx_k^*)- \frac{1}{k} \log \pi(\vx_k^*) \leqslant f(\vx^*) - \frac{1}{k} \log \pi(\vx^*),
$$
where the first inequality lies in $\pi(\vx) \leqslant B, \forall \vx \in \Omega$ and the second inequality lies in $\vx_k^* \in \arg \max m_k(\vx)$. Furtherly,  $\lim_{k \rightarrow \infty } f(\vx_k^*) =f^*$ follows from
$$
f^* \leqslant \liminf_{k \rightarrow \infty} f(\vx_k^*) 
 \leqslant \limsup_{k \rightarrow \infty} f(\vx_k^*) \leqslant \limsup_{k \rightarrow \infty} \left\{ f(\vx^*) - \frac{1}{k} \log \pi(\vx^*) \right \} =f^*. $$

\item [(b)] For any $\vx^* \in \Omega^* $, and $ \vx_k^* \in \Omega_k,$ we have $\pi(\vx_k^*) \geqslant \pi^* = \max_{\vx \in \Omega^*} \pi(\vx^*)$ for $k>0$, which follows from:
$$
f(\vx_k^*)-\frac{1}{k} \log \pi(\vx_k^*) \leqslant f(\vx^*)-\frac{1}{k}\log \pi(\vx^*) \leqslant f(\vx_k^*)-\frac{1}{k}\log \pi(\vx^*). 
$$

\item [(c)] The monotonicity of the sequence $\{ \pi(\vx_k^*)\}_{k=1, 2, \cdots}$ follows from:
\begin{align*}
f(\vx_{k+1}^*)-\frac{1}{k+1} \log \pi(\vx_{k+1}^*) 
& \leqslant f(\vx_k^*)-\frac{1}{k+1}\log \pi(\vx_k^*) \\
\iff f(\vx_{k+1}^*)-\frac{1}{k+1} \log \pi(\vx_{k+1}^*) & \leqslant f(\vx_{k}^*)-\frac{1}{k} \log \pi(\vx_k^*) + \frac{1}{k(k+1)} \log \pi(\vx_k^*) \\
\iff 
f(\vx_{k+1}^*)-\frac{1}{k+1} \log \pi(\vx_{k+1}^*) & \leqslant f(\vx_{k+1}^*)-\frac{1}{k}\log \pi(\vx_{k+1}^*) + \frac{1}{k(k+1)} \log \pi(\vx_k^*) \\
\iff
\log \pi(\vx_{k+1}^*) &\leqslant \log \pi(\vx_k^*)
\end{align*}
\end{enumerate}
\end{itemize}
\end{proof}

\section{Experimental Settings} \label{sec:supexp}
Consistent with the experimental parameters adopted by ZoRD \citep{shu2022zeroth}, we designate the domains of the Ackley and Levy functions as $[-20, 20]^d$ and $[-7.5, 7.5]^d$, respectively. For ProGO, the parameter for LSS is sample size as $N=200$ and burn-in period $n_b=20$. The configurations are uniformly applied across the RGF, PRGF, GD, and ZoRD algorithms to ensure fair comparisons, where the Adam optimizer \citep{KingBa15} is utilized with a fixed learning rate of 0.1 and exponential decay rates of 0.9 and 0.999. It should be noted that while ProGO is implemented in the R environment, other algorithms are executed in Python. Although this discrepancy may affect runtime comparisons, it does not influence accuracy results.

\end{document}